\documentclass[11pt,a4paper,reqno]{amsart}
\usepackage{mathrsfs,amsmath,amsfonts,amssymb,amsthm,graphicx,color,mathtools}
\numberwithin{equation}{section}

\newcommand{\p}{\partial}
 \renewcommand{\leq}{\leqslant}
\renewcommand{\geq}{\geqslant}
\renewcommand{\div}{\operatorname{div}}
\renewcommand{\O}{\Omega}

\newcommand{\tr}{\operatorname{tr}}

\def\({\left(}
\def\){\right)}
\def\[{\left[}
\def\]{\right]}
\def\l|{\left|}
\def\r|{\right|}
 \newtheorem{theorem}{{\bf Theorem}}[section]

\newtheorem{lemma}[theorem]{{\bf Lemma}}
\newtheorem{prop}[theorem]{{\bf Proposition}}
\newtheorem{remark}[theorem]{{\bf Remark}}
\def\grad{\nabla}
\DeclareMathOperator{\diverg}{div}
\newcommand{\Lsym}[1]{D(#1)}  
\newcommand{\Lskew}[1]{W(#1)}
\def\I{\mathbb{I}_d}  

\newcommand{\barg}[1]{\bigl(#1\bigr)}
\newcommand{\Barg}[1]{\Bigl(#1\Bigr)}
\newcommand{\bset}[1]{\bigl\{#1\bigr\}}

\newcommand{\bnorm}[1]{\bigl\|#1\bigr\|}
\newcommand{\Bnorm}[1]{\Bigl\|#1\Bigr\|}
\newcommand{\bbar}[1]{\bigl|#1\bigr|}

\newcommand{\bsqb}[1]{\bigl[#1\bigr]}

\newcommand{\dv}[1]{\,{\rm d}#1}
\newcommand{\bref}[1]{(\ref{#1})}
\newcommand{\ts}[1]{\Omega_{#1}}  

\newcommand{\tildeQ}{Q_0}

\def\weakly{\rightharpoonup}
\def\weaklystar{ \stackrel{\ast}{\rightharpoonup}}

\def\calF{\mathcal{F}}
\def\R{\mathbb{R}}

\def\BRXT{{B}_{X_T}(0,R)}
\def\BRXTzero{{B}_{X_{T,0}}(0,R)}

\def\H{L^2_\sigma(\Omega)}
\def\V{H^1_{0,\sigma}(\Omega)}

\newcommand{\dt}[1]{\partial_t#1} 

\def\epsilon{\varepsilon}

\def\Teps{T^{\epsilon}}

\begin{document}
\title[Beris--Edwards model]{Strong solutions for the \\
Beris-Edwards model for nematic liquid crystals with homogeneous Dirichlet boundary conditions}

\author{Helmut Abels}
\author{Georg Dolzmann}
\author{YuNing Liu}
\address{Fakult\"at f\"ur Mathematik, Universit\"at Regensburg, 93040 Regensburg, Germany.}
\email{helmut.abels@mathematik.uni-regensburg.de}
%
%

\date{\today}


\thanks{The third author gratefully acknowledges partial financial support
by the SPP 1506 ``Transport Processes at Fluidic Interfaces'' of the German Science
Foundation (DFG) through the grant AB285/4-2}

\begin{abstract}

Existence and uniqueness of local strong solution for the Beris--Edwards model for
nematic liquid crystals, which couples the Navier-Stokes equations with an evolution
equation for the $Q$-tensor, is established
on a bounded domain $\Omega\subset\mathbb{R}^d$ in the case of homogeneous Dirichlet
boundary conditions. The classical Beris--Edwards model is enriched by including
a dependence of the fluid viscosity on the $Q$-tensor.
The proof is based on a linearization of the system and Banach's fixed-point
theorem.

\end{abstract}

\maketitle

\section{Introduction}
Most of the classical models in the context of continuum mechanics
for the description of nematic liquid crystals fall into one of three major models
classes, the
Oseen-Frank model~\cite{Oseen1933,Zoecher1933,Frank1958},
the Ericksen-Leslie
model~\cite{EricksenARMA1962Hydrostatics,EricksenAdvLiqCryst1976,LeslieQJMAM1966,
LeslieARMA1968},
and the Beris-Edwards model~\cite{BerisEdwards1994}. The Oseen-Frank model
is a variational model in which the configurations of liquid crystals are
described by a director field $n\in\mathbb{S}^{d-1}$
and in which observed configurations are
explained as minimizers of a free energy functional. In the Ericksen-Leslie model,
the evolution of the director field is coupled with an evolution equation
for the underlying flow field which is given by the Navier-Stokes equation
with an additional forcing term. The most comprehensive model is the Beris-Edwards
model in which the director field is replaced by a $Q$-tensor
field~\cite{DeGennesProst1995}, thus allowing for a variable degree of order
in the material.
A detailed discussion of
this model and its connections to closely related models can be found
in~\cite{SonnetEtAl}, see
also~\cite{DennistonEtAlPRE2001,Hess1976,LubenskyStarkPRE2005,
OlmstedGoldbartPRA1990,SonnetEtAl,TothEtAlPRE2003},
and the literature therein.

In the $Q$-tensor models~\cite{DeGennesProst1995}, the unit director field $n$ is replaced by
a symmetric, traceless $d\times d$ tensor $Q$. This tensor is said to be uniaxial
if it has two equal non-zero eigenvalues and in this case it can be represented as
\begin{equation*}
  Q=s\left(n\otimes n-\frac 1d\,\I\right)\,,
\end{equation*}
where the scalar order parameter $s\in [-\frac 12,1]$  measures the
degree of orientational ordering. Connections between the director field $n$ and
a corresponding uniaxial $Q$-tensor field have been investigated
in~\cite{BallZarnescuMCLC2008,BallZarnescuARMA2011}.

In order to formulate the system of nonlinear partial differential equations
describing the liquid crystal flow with velocity $u$ and pressure $p$
and its orientation tensor $Q$,
we assume that $\Omega\subset\R^d$, $d=2,3$, is a bounded domain with boundary of class $C^4$, 
that $T>0$ is the time
horizon, that $\ts{T}=(0,T)\times\Omega$ is the time-space cylinder, and
that $a$, $b$, $c$, $\lambda$ and $\Gamma$ are positive constants.
As usual, the
material is supposed to be incompressible and the system is considered at
a fixed temperature which is not explicitly included in our notation.

The Beris-Edwards model leads to the following system, which contains the Navier-Stokes equations in $\ts{T}$
with variable viscosity and
an additional forcing term and an evolution
equation of parabolic type for the order parameter $Q$,
\begin{align}\label{strongformeqn}
\begin{split}
 \dt{u} + (u\cdot \grad)u  + \grad p
&\, = \diverg\barg{\nu(Q)\Lsym{u} }+ \diverg\barg{\sigma(Q,Q)+ \tau(Q)}\,,\\
\diverg u &\, = 0\,, \\
 \dt{Q} + \barg{  u \cdot \grad } Q - S(\grad u, Q) &\, = \Gamma H(Q)
 \end{split}
\end{align}
subject to the initial and boundary conditions,
\begin{alignat}{4}\label{strongformeqnbc}
 Q|_{t=0} &\, = Q_0&&\text{ in }\Omega\,,
\qquad & Q\bigr|_{\partial\Omega\times(0,T)}&\, = 0\,, \\\label{strongformeqnbc'}
 u|_{t=0} &\, = u_0&&\text{ in }\Omega\,,\qquad & u\bigr|_{\partial\Omega\times(0,T)}&\, = 0\,,
\end{alignat}
with tensors $\sigma$, $\tau$, and $S$ given by
\begin{align}\label{tensors}
\begin{split}
 \sigma(Q_1,Q_2) &\, = Q_1 \Delta Q_2 - (\Delta Q_2) Q_1\,,
\quad \tau(Q) \, =-\lambda \grad Q\odot \grad Q \,,\\
 S(\nabla u, Q) &\, = S(\Lskew{u} , Q) = \Lskew{u}  Q - Q\Lskew{u} \,,
\end{split}
\end{align}
where
\begin{align}\label{Ldecomposition}
\Lsym{u}  = \frac{1}{2}\barg{ \nabla u + (\nabla u)^T }\,,\quad
 \Lskew{u} = \frac{1}{2}\barg{\nabla u - (\nabla u)^T}
\end{align}
are the stretch and the vorticity tensor, respectively.
The forcing term in the evolution of the order parameter is given by
\begin{align}\label{strongforcing}
 H=H(Q) &\, =\lambda\Delta Q - aQ + b\barg{Q^2 - \frac{1}{d}\,\tr(Q^2)\I} - c\tr(Q^2) Q\,,
\end{align}
and we shall denote by $L$ the lower-order terms in \eqref{strongforcing}, i.e.,
\begin{align}\label{loworder}
 L=L(Q) &\, = - aQ + b\barg{Q^2 - \frac{1}{d}\,\tr(Q^2)\I} - c\tr(Q^2) Q\, .
\end{align}
Note that \eqref{strongforcing} is
related to the variational derivative of the free energy functional
 which uses
the one-constant approximation for the Oseen-Frank energy of liquid crystals
together with a Landau-DeGennes expression for the bulk energy,
\begin{align}\label{laudau}
 \calF(Q) = \int_\Omega \barg{\frac{\lambda}{2}\, |\nabla Q|^2 + f_B(Q)}\dv{x}
\end{align}
where the bulk energy $f_B$ is given by
\begin{align*}
 f_B(Q) = \frac{a}{2}\,\tr(Q^2) -\frac{b}{3}\,\tr(Q^3)
+\frac{c}{4}\,\tr(Q^4)  \,.
\end{align*}
We see from \eqref{loworder} that
\begin{equation}\label{lowerder}
  L+\frac bd\tr(Q^2)\I=-\nabla_Q f_B(Q).
\end{equation}
The constitutive assumptions~\bref{tensors} and~\bref{strongforcing} are special cases of
more general expressions~\cite{DennistonEtAlPRE2001} for the corresponding
tensors and do not include
alignment effects due to the flow. They correspond to the assumption that
the coupling parameter $\xi$ between the order tensor $Q$ and the stretch
tensor $\Lsym{u}$ is zero.

It is known that the viscosity of a liquid crystal may depend on its local orientation
with respect to the fluid flow or on the flow
rate~\cite{DeGennesProst1995,DennistionEtAlCTPS2001}. Moreover,
the classical derivation of the constitutive equations~\cite{LeslieARMA1968}
identifies the viscous stress tensor as a sum of contributions each of which has
its own hydrodynamic viscosity coefficient. The temperature and order parameter
dependence of these nematic viscosities was, e.g., discussed
in~\cite{DiogoMartinsJPhysique1982}. Mathematically, it is a challenging
task to include the full dependence of the viscosity coefficients on the
order parameter. As a first step towards this goal, we include a
dependence of the fluid viscosity $\nu$ on the order parameter $Q$ via
$\nu=\nu(Q)$. Even with this very weak coupling of the director field with
the fluid viscosity, the analysis requires the assumption
\begin{equation}\label{viscosity}
  \nu\in C^2(\mathbb{R}^{d\times d}),\quad 0<c_0\leq \nu(\cdot)\leq c_1<\infty
\end{equation}
and leads to a significant number of additional terms in our estimates. More
general dependencies have, e.g., been explored in~\cite{CaldererLiuSIAMJApplMath2000}.

Global weak solutions for the system~\bref{strongformeqn} -- \bref{strongforcing}
with constant viscosity $\nu$ and $\Omega=\R^d$ were constructed
in~\cite{PaicuZarnescuARMA2012}. The full system with coupling parameter $\xi\neq 0$ but
sufficiently small is treated in~\cite{PaicuZarnescuSIMA2011} if $\Omega=\R^d$.
 In \cite{Wilkinson} Wilkinson studied the system \eqref{strongformeqn}-\eqref{tensors} under periodic boundary conditions in the case that $f_B$ is replaced by a certain singular potential, which guarantees that $Q$ attains only physically reasonable values. For general $\xi$ existence of weak solutions was established. Moreover, he proved higher regularity in the case of two space dimensions and $\xi=0$. Finally, Feireisl et al.~\cite{FeireislEtAlQTensor} derived a non-isothermal variant of the Beris-Edwards system and proved existence of weak solutions for this system in the case of a singular potential and for periodic boundary conditions. Recently, Wang et al.\ establish in \cite{wangzz} a rigorous convergence result from the Beris-Edwards system to the Ericksen-Leslie system, which is widely investigated in the literature. 
First results in the case of a bouded domain were obtained  by Guill\'en-Gonz\'alez and Rodr\'iguez-Bellido~\cite{GuillenBellidoWeak} in the case $\xi=0$, where existence of weak solutions in bounded domains with inhomogeneous Dirichlet and homogeneous Neumann boundary conditions was proved. Moreover, the authors proved a Serrin-type uniqueness criterion. In \cite{ADL} the authors prove existence of weak solutions and well-posedness locally in time with higher regularity in time in the case of inhomogeneous mixed Dirichlet/Neumann boundary conditions in a bounded domain.

The main novelty of the present contribution concerns short time existence for strong solutions in bounded domains in the case of homogeneous Dirichlet/Neumann boundary conditions:
\begin{theorem}\label{maintheorem}
For any $u_0\in H^1_{0,\sigma}(\O)$ and $Q_0\in H^1_0(\O;\mathbb{S}_0)\cap H^2(\O;\mathbb{S}_0)$,
there exists some $T>0$ and a unique solution $(u,Q)$
of the system~\eqref{strongformeqn} with
\begin{align*}
 u\in H^1(0,T;L^2_{\sigma}(\O))\cap L^2(0,T;H^2(\O;\R^d))\,,\
Q\in H^1(0,T;H^1_0(\O;\mathbb{S}_0))\cap L^2(0,T;H^3(\O;\mathbb{S}_0))
\end{align*}
satisfying the initial and
boundary conditions~\eqref{strongformeqnbc}-\eqref{strongformeqnbc'}.
\end{theorem}

The most subtle point in our analysis is related to the fact, that the
evolution equation for the director field $Q$ implies in view of the
regularity of the strong solutions the compatibility condition
$\Delta Q(t, \cdot)=0$ on $(0,T)\times \partial\Omega$.
This observation leads to the following outline for the proof of
Theorem~\ref{maintheorem}, see Section~\ref{prelim} for the notation
used throughout the paper. We apply (formally) the Helmholtz projector to the first equation
of~\eqref{strongformeqn}, and obtain 
\begin{equation*}
 \dt{u}-  P_{\sigma}\div(\nu(Q) \Lsym{u} )
=P_{\sigma}\diverg \barg{\tau(Q)+\sigma(Q,Q)}-P_{\sigma}(u\cdot\nabla u)\,,
\end{equation*}
an equation in which the pressure has been eliminated
from the system. The existence of a local solution
for the corresponding system is now obtained by applying Banach's fixed-point
theorem to the nonlinear operator $\mathscr{L}=\mathcal{L}^{-1}\mathcal{N}$
which is constructed from a linearization
of the system about the initial value $\tildeQ$, i.e.,
the system is rewritten in the form
\begin{align*}
 \mathcal{L}(\tildeQ)(u,Q) = \mathcal{N}(\tildeQ)(u,Q)
\end{align*}
with $\mathcal{L}$ defined in~\bref{linearization} and  $\mathcal{N}$
in~\bref{nonlinear} below. We prove in Section~\ref{linearsystem} that $\mathcal{L}$
is bounded, onto and one-to-one between suitable spaces and in Section~\ref{nonlinearoperator}
that $\mathcal{N}$ is Lipschitz continuous with arbitrarily small Lipschitz constant for
$T$ sufficiently small.
In order to deal with the compatibility
condition $\Delta Q=0$ on $(0,T)\times \partial\Omega$, which we need to impose also
on solutions of the linearized equation, we add a singular perturbation
to the operator in the third equation in~\bref{strongformeqn},
i.e., we consider an approximation of the heat operator $\dt{Q}-\Delta Q$
in the linearization of the evolution equation for the director field by the operator
\begin{align*}
 \dt{Q}+\varepsilon\Delta^2 Q-\Delta Q\,,\quad \epsilon>0\,,
\end{align*}
see~\bref{approstrong} for the complete set of equations. The proof of global
existence of solutions to this approximating system is quite classical in the
sense that we employ a fixed-point argument to obtain a local solution and
global a~priori estimates which are closely related to the energy law which
holds for the system~\bref{strongformeqn}. This analysis is presented
in Section~\ref{singularperturbation}. Uniform bounds in $\epsilon>0$ allow us
to construct global solutions to the linear system in Section~\ref{linearsystem}
by passing to the limit $\epsilon\searrow 0$.
With all these results in place, the
proof of Theorem~\ref{maintheorem} is given in Section~\ref{maintheoremproof}.
Finally, in Section~\ref{sec:Global} we discuss estimates of the solutions, which are the basis to prove existence of strong solutions globally in time for small initial velocities and $Q_0$ close to the global minimizer of $\mathcal{F}$ if $d=3$ and any sufficiently regular initial data if $d=2$, cf. Lin and Liu~\cite{LinLiu1995} or Wu et al~\cite[Lemma~3.3 and Corollary~4.2]{wu}.





\section{Preliminaries}\label{prelim}
In this section we collect the relevant definitions and some auxiliary results
which will be used throughout the paper.

\subsection{Notation}
For two vectors $a$, $b\in\R^d$ we set $a\cdot b = \sum_{i=1}^d a_i b_i$
and for two matrices $A$, $B\in\R^{d\times d}$ we set $A: B
= \sum_{i,j=1}^d A_{ij} B_{ij}= \tr (A^T B)$. Then
\begin{align}\label{linalg}
 (AB):C=B:(A^TC)=A:(CB^T)\quad
\text{ for all }A,\, B,\, C\in \R^{d\times d}
\end{align}
and we omit the parentheses for simplicity in the following
if it is clear from the context that the equation is scalar.
 We shall use
$(\cdot,\cdot)_{H}$ to denote the inner product in a Hilbert space $H$ and use $\langle\cdot,\cdot\rangle_{X',X}$ to denote the dual product between $X$ and its dual space $X'$.
The symbol $C$ denotes a generic constant whose value may change from line to line. We do not indicate the dependence of constants on the domain $\Omega$.
In order to specify the dependence of estimates on key inequalities, some
constants are specifically labeled, e.g., $C_{\mathcal{S}}$ denotes
the constant in the  $H^2$-estimate for the Stokes operator in~\eqref{stokregular}
below and $C_{\mathcal{E}}$ stands for the constant in the
$H^2$-estimate for the Laplace operator,
  \begin{equation}\label{ellipticest}
    \|\Delta f\|_{L^2(\O)}
\geq C_{\mathcal{E}}\|f\|_{H^2(\O)}, \quad f\in H^1_0(\O)\cap H^2(\Omega)\,.
  \end{equation}
Let $\mathbb{S}_0\subset \mathbb{R}^{d\times d}$ denote the space of $Q$-tensors,
i.e.,
\begin{equation*}
  \mathbb{S}_0=\{Q\in\mathbb{R}^{d\times d},\,Q=Q^T,\,\tr  Q=0\}\,.
\end{equation*}
The norm of a matrix $F\in \mathbb{R}^{d\times d}$ is given
by $|F|^2=\tr(F^T F)=F:F$ and hence it follows for all
$Q\in \mathbb{S}_0$ that $|Q|^2=\tr(Q^2)$.
The symbol $\I$ denotes the identity matrix in dimension $d$.

Throughout the paper we adopt Einstein's summation convention.
Moreover, we define the contraction $\nabla Q\odot \nabla Q$ for a second order
tensor field $Q=(Q_{ij})_{i,j=1}^d$ by
\begin{align*}
\barg{\nabla Q\odot \nabla Q}_{ij}
= \sum_{k,\ell=1}^d \partial_i Q_{k\ell}\partial_j Q_{k\ell}
=\partial_i Q_{k\ell}\partial_j Q_{k\ell} \,.
\end{align*}
Furthermore $a\otimes b:= ab^T$ for $a,b\in\R^d$. For matrix-functions $F$, we denote $(\div F)_i:=\partial_{j}F_{ij}$.

Finally, we note that we will offen omit ``$\mathrm{d} x$, $\mathrm{d} t$,...'' at the end of the integrals in order to obtain shorter formulas.

\subsection{Function spaces}
We use standard notation for Lebesgue and Sobolev spaces
of scalar and vector valued functions and we omit the
domain and the range in the notation if it is clear from the context.
The subscript $\sigma$ indicates solenoidal vector fields, e.g.,
\begin{align*}
 C^{\infty}_{0,\sigma}(\O)
=\bset{u\in C^{\infty}_{0}(\O;\R^d),\,\diverg u = 0}
\end{align*}
and
\begin{align*}
 L^2_\sigma(\Omega) 
&=\bset{u\in L^{2}(\O;\R^d),\,\diverg u = 0,\,\gamma(u)=0}=\overline{C_{0,\sigma}^\infty(\Omega)}^{L^2(\Omega;\R^d)}\,,\\
 H^1_{0,\sigma}(\Omega) 
&=\bset{u\in H^{1}_0(\O;\R^d),\,\diverg u = 0}\,,\quad H^{-1}_\sigma(\Omega)= (H^1_{0,\sigma}(\Omega))'
\end{align*}
where $\gamma(u)=u\cdot n\in H^{-\frac12}(\partial\Omega)$ is understood in the sense of a (weak) trace and where $n$
is the exterior normal to $\partial\Omega$.
Note that
\begin{align*}
 L^2(\Omega;\R^d) = \H \oplus \H^\perp \,\text{ with }\,
\H^\perp = \bset{ u \in L^2(\Omega;\R^d),\,u=\nabla q
\text{ for some }q\in H^1(\Omega) }\,.
\end{align*}
The orthogonal projection $L^2(\Omega;\R^d) \to \H$
is denoted by $P_\sigma$ and often referred to as the   Helmholtz projection and the
decomposition
\begin{equation}
  \label{eq:Helmholtz}
u=P_\sigma u + \nabla q\qquad \text{with }q\in H^1(\Omega)
\end{equation}
 for $u\in L^2(\Omega;\R^d)$ as Helmholtz decomposition.
See~\cite{Temam1984} for more information on these spaces.

In order to write~\eqref{strongformeqn} in an abstract way as an
operator equation
between two Banach spaces, we  define the domain $X_T  = X_{T;u}\times X_{T;Q}$ and
the range $Y_T = Y_{T;u}\times Y_{T;Q}$ by
\begin{equation*}
\begin{split}
    X_{T;u}&=
L^2(0,T;H^2(\O;\R^d)\cap H^1_{0,\sigma}(\O))
\cap H^1(0,T;\H),
\\
X_{T;Q}&=\left\{Q\in L^2(0,T;H^3(\O;\mathbb{S}_0))
\cap H^1(0,T;H^1_0(\O;\mathbb{S}_0)),\Delta Q|_{\partial\Omega}=0 \right\},\\
  Y_T&=L^2(0,T;\H)\times L^2(0,T;H^1_0(\O;\mathbb{S}_0))
\end{split}
\end{equation*}
together with the norms defined by 
\begin{equation*}
  \begin{split}
     \|(u,Q)\|_{X_T}^2 &=
\int_0^T \barg{ \|u(t)\|_{H^2(\Omega)}^2 + \|\dt{u}(t)\|_{L^2(\Omega)}^2 }\dv{t}
+
\int_0^T \barg{ \|Q(t)\|_{H^3(\Omega)}^2 + \|\dt{Q}(t)\|_{H^1(\Omega)}^2 }\dv{t}\\
&\qquad +\|u|_{t=0}\|_{H^1(\O)}^2+\|Q|_{t=0}\|_{H^2(\O)}^2
  \end{split}
\end{equation*}
and
\begin{align*}
  \|(f,G)\|_{Y_T}^2 =
\int_0^T \barg{ \|f(t)\|_{L^2_\sigma(\Omega)}^2 + \|G(t)\|_{H^1(\Omega;\mathbb{S}_0)}^2 }\dv{t}\,.
\end{align*}
To formulate the initial conditions we set
\begin{align*}
   Y_0&= H^1_{0,\sigma}(\O)\times \barg{H^2(\O;\mathbb{S}_0)\cap H^1_0(\O;\mathbb{S}_0)}\,.
\end{align*}
For any Banach space $Z$, $z\in Z$ and $R>0$ we
denote by ${B}_Z(z,R)$ the closed 	
ball of radius $R$ about $z$ with respect to the norm of $Z$,
\begin{equation*}
 {B}_Z(z,R) =\{z\in Z,\,\|z\|_Z\leq R\}\,.
\end{equation*}

\subsection{Inequalities and a~priori estimates}

The following estimates will be frequently used.

\begin{lemma}
There are constants $C>0$ such that:
\begin{itemize}
 \item [(i)] If $d=2$, then for all $u\in H^1_0(\O)$,
  \begin{equation*}
    \|u\|_{L^4(\O)}
\leq C\|u\|^{\frac 12}_{L^2(\O)}\|\nabla u\|^{\frac 12}_{L^2(\O)}\,.
  \end{equation*}

 \item [(ii)] If $d=3$, then for all $u\in H^1_0(\O)$,
  \begin{equation}\label{inter2}
     \|u\|_{L^4(\O)}
\leq C\|u\|^{\frac 14}_{L^2(\O)}\|\nabla u\|^{\frac 34}_{L^2(\O)}\,,
  \end{equation}
  and
    \begin{equation}\label{inter3}
     \|u\|_{L^3(\O)}
\leq C\|u\|^{\frac 12}_{L^2(\O)}\|\nabla u\|^{\frac 12}_{L^2(\O)}\,.
  \end{equation}
If $u\in H^1_0(\O)\cap H^2(\Omega)$, then
  \begin{equation}\label{inter4}
    \|u\|_{L^{\infty}(\O)}
\leq C\|u\|_{H^1(\O)}^{\frac 12}\|u\|_{H^2(\O)}^{\frac 12}\,.
  \end{equation}
Finally, if $f\in H^1(\O)$ and $g\in H^2(\O)$, then
    \begin{equation}\label{important}
    \|fg\|_{H^1(\O)}
\leq C\|f\|_{H^1(\O)}\|g\|^{\frac 12}_{H^1(\O)}\|g\|_{H^2(\O)}^{\frac 12}\,.
  \end{equation}
\end{itemize}
\end{lemma}
\begin{proof}
  Inequalities (i) and~\eqref{inter2} can be found in
\cite[pp. 291 and 296]{Temam1984} and
\eqref{inter3} can be obtained from~\eqref{inter2} with
H\"{o}lder's inequality in its interpolation form.
For~\bref{inter4} see, e.g.,~\cite[Lemma 4.10]{ConstantinFoias1988}. To
prove~\bref{important} it suffices to estimate the gradient term
by~\bref{inter3} and~\bref{inter4}
\begin{align*}
 \| \nabla(fg) \|_{L^2(\O)}&\,
\leq  \| f\,\nabla g \|_{L^2(\O)} +  \| (\nabla f)\,g \|_{L^2(\O)}
\leq \| f\|_{L^6(\O)} \| \nabla g \|_{L^3(\O)} +
\| \nabla f \|_{L^2(\O)} \| g\|_{L^\infty(\O)} \\
&\,\leq C
\| f\|_{H^1(\O)} \| \nabla g \|_{L^2(\O)}^{1/2} \| \nabla g \|_{H^1(\O)}^{1/2}+
C \| \nabla f \|_{L^2(\O)}
\| g \|_{H^1(\O)}^{1/2} \|  g \|_{H^2(\O)}^{1/2}\,.
\end{align*}
The proof is now complete.
\end{proof}
\begin{remark}
  By the classical Gagliardo-Nirenberg inequality, we have in two and three dimensions
  \begin{equation}\label{inter33}
     \|u\|_{L^4(\O)}
\leq C\|u\|^{\frac 14}_{L^2(\O)}\|u\|^{\frac 34}_{H^1(\O)},
  \quad
     \|u\|_{L^3(\O)}
\leq C\|u\|^{\frac 12}_{L^2(\O)}\| u\|^{\frac 12}_{H^1(\O)}\,
  \end{equation}
respectively, 
for all $u\in H^1(\Omega)$.
\end{remark}
In the following we will frequently use that
  \begin{equation}\label{interpo2}
  \|f\|^2_{C([0,T];H)}\leq 2(\|f\|_{H^1(0,T;H_1')}\|f\|_{L^2(0,T;H_1)}+\|f(0,\cdot)\|^2_{H})
  \end{equation}
for all $f\in H^1(0,T;H_1')\cap L^2(0,T;H_1)$, where $H_1\hookrightarrow H\cong H' \hookrightarrow H_1'$ is a Gelfand-triple.

The construction of solutions which satisfy the compatibility condition
$\Delta Q=0$ on $\partial\Omega$ requires the approximation of the equation
for $Q$ by a fourth-order equation. The existence of solutions can be inferred from
results in semi-group theory.

\begin{prop}\label{abstrac}
  Let $\mathcal{A}:\mathcal{D}(\mathcal{A})\subset H\to H$ be a generator
of a bounded analytic semi-group on a Hilbert space $H$ and let $1<q<\infty$.
Then for every $f\in L^q(0,\infty;H)$ and
$u_0\in \barg{H,\mathcal{D} ( \mathcal{A} )}_{1-\frac1q, q}$
there is a unique $u:[0,\infty)\to H$ such that
$\frac{\displaystyle du}{\displaystyle dt},\mathcal{A}u\in L^q(0,\infty;H)$ solving
\begin{equation*}
  \begin{split}
      \frac{du}{dt}(t)-\mathcal{A}u(t)&=f(t),~t>0,\\
  u(0)&=u_0.
  \end{split}
\end{equation*}
Moreover, there is a constant $C_q>0$ independent of $f$ and $u_0$ such that
\begin{equation*}
\Bnorm{\frac{du}{dt}}_{L^q(0,\infty;H)}
+\|\mathcal{A}u\|_{L^q(0,\infty;H)}
\leq
 C_q\barg{\|f\|_{L^q(0,\infty;H)} +
\|u_0\|_{(H,\mathcal{D}(\mathcal{A}))_{1-\frac 1q,q}}
}\,.
\end{equation*}
\end{prop}
\begin{proof}
In the case $u_0=0$ the statement is the main result of \cite{DeSimonMaxReg}. The general case can be easily reduced to the case $u_0=0$ by subtracting a suitable extension. The existence of such an extension follows e.g. from \cite[Chapter~III, Theorem 4.10.2]{Amann}.
\end{proof}

As an application, we obtain an existence result for the fourth order equation
in the subsequent lemma.

\begin{lemma}[Strong solutions for a 4th-order parabolic system]\label{para}
~\\
 Let $\varepsilon>0$, $T\in (0,T_0)$ be fixed constants.
For any $G\in L^2(\ts{T};\mathbb{S}_0)$ and $Q_0\in H^2(\O;\mathbb{S}_0)\cap H^1_{0}(\O;\mathbb{S}_0)$,
the system
\begin{equation}
\begin{alignedat}{2}
\dt{Q}+  \varepsilon \Delta ^2Q -\Delta Q &\, =G
\quad&&\text{ in }(0,T)\times \Omega\,,\\
Q|_{\partial\Omega}=\Delta Q|_{\partial\Omega}&\,=0&&\text{ on }(0,T)\times \partial\Omega\,,\\
Q|_{t=0}&\,=Q_0&&\text{ in }\Omega\,
\end{alignedat}
\end{equation}
has a unique strong solution
$Q\in L^2(0,T;H^4(\O;\mathbb{S}_0)\cap H^1_{0}(\O;\mathbb{S}_0))\cap H^1(0,T;L^2(\O;\mathbb{S}_0))$.
Moreover, the following estimate holds
\begin{equation*}
  \|Q\|^2_{L^2(0,T;H^4)}+\|\dt{Q}\|^2_{L^2(\ts{T})}
\leq C(\varepsilon)(\|Q_0\|^2_{H^2(\O)}+\|G\|^2_{L^2(\ts{T})})\,.
\end{equation*}
\end{lemma}

\begin{proof}
 Let $\mathcal{A}=-\varepsilon\Delta^2+\Delta$. We first show that
 \begin{equation*}
   \mathcal{A}\colon\mathcal{D}(\mathcal{A})
=\left\{Q\in H^4(\O;\mathbb{S}_0);~Q|_{\p\O}=0,~\Delta Q|_{\p\O}=0\right\}\to L^2(\O;\mathbb{S}_0)
 \end{equation*}
 generates a bounded analytic semi-group. By standard elliptic regularity theory,
the operator equation $\mathcal{A}u=f$ has a unique solution
$u\in \mathcal{D}(\mathcal{A})$ for each $f\in L^2(\O;\mathbb{S}_0)$.
Moreover, $\mathcal{A}$ is a symmetric operator and negative. By standard theory of elliptic equations $\lambda - \mathcal{A}$ is invertible for all $\lambda>0$ sufficiently large. Hence $A$ is self-adjoint and has  negative spectrum.
Thus it generates an analytic semi-group. Finally, $(L^2(\O;\mathbb{S}_0),\mathcal{D}(\mathcal{A}))_{\frac 12,2}=H^2(\O;\mathbb{S}_0)\cap H^1_0(\O;\mathbb{S}_0)$ is a consequence of \cite[Th\'{e}or\`{e}me 8.1]{Grisvard} applied component-wise.
Hence
Proposition~\ref{abstrac} implies the statement of the lemma.
\end{proof}

For completeness we also quote existence results for
the Stokes system with a prescribed but variable viscosity,
\begin{equation}\label{stokvis}
\begin{alignedat}{2}
\dt{u}-\div\barg{\nu(\tildeQ)\Lsym{u} }+\nabla p &\, =f
\quad&&\text{ in }(0,T)\times \Omega\,,\\
\diverg  u &\,=0 &&\text{ in }(0,T)\times \Omega\,,\\
u|_{\partial\Omega}&\,=0&&\text{ on }(0,T)\times \partial\Omega\,,\\
u|_{t=0}&\,=u_0&&\text{ in }\Omega\,.
\end{alignedat}
\end{equation}


The following result establishes the regularity of the Stokes operator
with variable viscosity.

\begin{lemma}[Regularity of Stokes operator]~\\
\label{Stokesregularity}Let $\nu$ satisfy~\eqref{viscosity} and suppose that
 $Q_0\in H^2(\O;\mathbb{S}_0)\cap H^1_{0}(\O;\mathbb{S}_0)$.
Let $u\in H^1_{0,\sigma}(\O)$ be such that
  \begin{equation*}
    \barg{\nu(\tildeQ)\Lsym{u} ,\Lsym{\varphi}}_{L^2(\O)}=
\barg{f,\varphi} _{L^2(\O)},\quad \forall\varphi\in C^{\infty}_{0,\sigma}(\O),
  \end{equation*}
  where $f\in L^2(\O;\R^d)$. Then $u\in H^2(\O;\R^d)$ and there exists a constant
$C=C\barg{ \|\tildeQ\|_{H^2(\O)},\nu}$ independent of $f$ such that
  \begin{equation}\label{regstok}
    \|u\|_{H^2(\O)}\leq C\|f\|_{L^2(\O)}.
  \end{equation}
\end{lemma}
 A proof can be found in \cite[Lemma 4]{Abels2009} in the case that $\nu$ depends on a scalar quantity $c$. But the proof directly carries over to the present situation.

\begin{remark}
Let  $A=-P_\sigma\div\barg{\nu(Q_0)\Lsym{u} }\colon \mathcal{D}(A)\subseteq \H\to \H$ be the Stokes operator with
prescribed viscosity, where $\mathcal{D}(A)= H^2(\Omega;\R^d)\cap \V$.
Then~\eqref{regstok} implies the following $H^2$-estimate for
the Stokes operator,
\begin{equation}\label{stokregular}
\|u\|^2_{H^2(\O)}\leq C_{\mathcal{S}}  \int_{\O}|Au|^2\,\mathrm{d}x
\,,
\end{equation}
where $C_{\mathcal{S}}=C_{\mathcal{S}}(\nu,\tildeQ)$ is independent of $u\in \mathcal{D}(A)$.

We will also extend $A=-P_\sigma\div\barg{\nu(\tildeQ)\Lsym{u} }$ to a map from $\V$ to $H^{-1}_\sigma(\Omega)$ defined by
\begin{equation*}
  \left\langle Au,\varphi \right\rangle_{H^{-1}_{\sigma},H^1_{0,\sigma}} = \int_\Omega \nu(\tildeQ)Du:D\varphi \, dx\qquad \text{for all}\ \varphi\in \V.
\end{equation*}
\end{remark}
Similarly to the case of constant viscosity, a weak solution of \eqref{stokvis} will be some  $u\in L^2(0,T;H^1_{0,\sigma}(\O))\cap H^1(0,T;H^{-1}_\sigma(\O))$ such that
\begin{alignat*}{2}
  \partial_t u + Au &= P_\sigma f&\quad&\text{in } L^2(0,T;H^{-1}_\sigma(\O)),\\
u|_{t=0} &= u_0 &\quad& \text{in } \H,
\end{alignat*}
where we note that $L^2(0,T;\V)\cap H^1(0,T;H^{-1}_\sigma(\O))\hookrightarrow C([0,T];\H)$ since $\V,\H$, and $H^{-1}_\sigma(\Omega)$ form a Gelfand triple.

The following result is concerned with existence and regularity of solutions
to~\eqref{stokvis}.

\begin{lemma}\label{stok}
  Let $\nu$ satisfy~\eqref{viscosity}, suppose that
 $Q_0\in H^2(\O;\mathbb{S}_0)\cap H^1_0(\O;\mathbb{S}_0),$ and
  let $u_0\in L^2_{\sigma}(\O)$, $f\in L^2(0,T;H^{-1}_\sigma(\Omega))$, where $T\in (0,\infty)$.
Then there is a unique weak solution $u\in L^2(0,T;H^1_{\sigma}(\O))\cap H^1(0,T;H^{-1}_\sigma(\O))$
to~\eqref{stokvis} satisfying  
  \begin{equation*}
    \|\dt{u}\|_{L^2(0,T;H^{-1}_\sigma(\Omega))}+\|u\|_{L^2(0,T;H^1_{\sigma}(\O))}
\leq C(\nu,\tildeQ,T)\barg{\|u_0\|_{L^2(\O)}+\|f\|_{L^2(0,T;H^{-1}_\sigma(\Omega))}}\,.
  \end{equation*}
 If, additionally, $u_0\in H^1_{0,\sigma}(\O)$ and $f\in L^2(0,T;L^2_{\sigma}(\O))$,
then $$u\in H^1(0,T;L^2_{\sigma}(\O))\cap L^2(0,T;H^2(\O;\mathbb{R}^d))$$  along with the estimate
  \begin{equation}\label{regustok}
    \|(\dt{u},\nabla^2 u )\|_{L^2(0,T;L^2(\O))}
\leq C(\nu,\tildeQ,T)
\barg{ \|u_0\|_{H^1_{0,\sigma}(\O)}+\|f\|_{L^2(0,T;L^2_{\sigma}(\O))}}\,.
  \end{equation}
Furthermore, $C(\nu,\tildeQ,T)$ can be chosen as an increasing function with respect to $T$ in both estimates.
\end{lemma}

\begin{proof}
 The proof is an immediate consequence of~\cite[Theorem 7, Lemma 6, and Proposition 4]{Abels2009}. In these results the viscosity $\nu$ depends on a scalar parameter $c\in BUC([0,T];W^1_r(\O))$ for some $r>d$. But the proofs directly carry over to the case that $c\in BUC([0,T];W^1_r(\Omega;\R^N))$ for any $N\in\mathbb{N}$ provided $\nu\in C^2_b(\R^N;\R)$ is bounded from below by a positive constant. Formally, one could also replace $c$ by $\nu(Q)$ and replace $\nu$ by an auxiliary function, which is the identity on $[c_0,c_1]$. Finally, the monotone dependence of $C(\nu,\O,\tildeQ,T)$ on $T>0$ can be easily verified by extending $f$ by zero for $t>T$ and applying the result on a time interval $(0,T')$ with $T'>T$.
\end{proof}


Finally, by the Sobolev embedding theorem, we deduce 
that
$\|Q\|_{L^{\infty}(\O)}\leq C\|Q\|_{H^2(\O)}$.
More generally, whenever $Q\in H^2(\O)$, then~\eqref{viscosity}
implies that
\begin{align}\label{boundvis}
     \|(\nabla \nu)(Q)\|_{L^{\infty}(\O)}+\|(\nabla^2\nu)(Q)\|_{L^{\infty}(\O)}
\leq C\barg{ \|Q\|_{H^2(\O)} }\,.
\end{align}


\subsection{An algebraic identity}
The following algebraic identity will be used later.

\begin{lemma}
Let  $Q,\,G\in L^1(\O)$
be symmetric tensors and $u\in W^{1,1}(\O)$. Then for a.e.~$x\in \Omega$,
\begin{align}\label{tensor4}
 S(\nabla u, Q):G = \nabla u: (GQ-QG) \,.
\end{align}
\end{lemma}

\begin{proof}
We use~\bref{linalg} and the symmetries to obtain
\begin{align*}
2 S(\nabla u, Q):G
&\, = \nabla u Q: G
-(\nabla u)^T Q: G
-Q  \nabla u:  G
+Q (\nabla u)^T  :G \\
&\, = 2\barg{\nabla u :G Q -  \nabla u : QG   }\,.
\end{align*}
This proves the assertion.
\end{proof}
We note that the lemma implies $\tr(S(\nabla u,Q))=S(\nabla u,Q):I=0$. Therefore $S(\nabla u,Q)\in \mathbb{S}_0$ since $S(\nabla u,Q)$ is symmetric.


\section{Existence for the linear system with singular perturbation}
\label{singularperturbation}
The key ingredient in the proof of the good mapping properties of the linearized
operator in~\bref{linearization} below is the  existence of solutions for
the following system including a singular perturbation. Suppose that
$\epsilon>0$ and consider the linear system
\begin{equation}\label{approstrong}
\begin{alignedat}{2}
\dt{u}-\div\barg{ \nu(\tildeQ)\Lsym{u}  }
+\nabla p-\diverg\sigma(\tildeQ,  Q)
&\,=F\qquad &&\text{ in }(0,T)\times \Omega\,,\\
\diverg  u&\,=0&&\text{ in }(0,T)\times \Omega\, ,\\
\dt{Q}+\varepsilon\Delta^2 Q-\Delta Q-S(\nabla u, \tildeQ)
&\,=G&&\text{ in }(0,T)\times \Omega\,,\\
u|_{\partial\O}&\,=0 &&\text{ on }(0,T)\times\partial\O\,,\\
 Q|_{\partial\O}=\Delta Q|_{\partial\O}&\,=0&&
\text{ on }(0,T)\times\p\O\,,\\
 (u,Q)|_{t=0}&\,=(u_0,Q_0)\quad && \text{ on }\Omega\,.
 \end{alignedat}
\end{equation}

\begin{prop}\label{approximate1}
Let $T>0$ be fixed. Then there exists an $\epsilon_0=\epsilon_0(T)>0$ such that
for all $\varepsilon\in(0, \epsilon_0)$ and for
all $(F,G,u_0,Q_0)\in Y_T\times Y_0$ the system~\eqref{approstrong}
has a unique solution
  \begin{equation}\label{regularity}
    (u,Q)\in X_T\text{ with } Q\in L^2(0,T;H^4(\O;\mathbb{S}_0))\,.
  \end{equation}
Moreover,   the following estimate holds
  \begin{equation}\label{bound}
    \|(u,Q)\|_{X_T}+\sqrt{\varepsilon}\|\Delta^2 Q\|_{L^2(\O_T)}\leq C(\tildeQ,T)\|(F,G,u_0,Q_0)\|_{Y_T\times Y_0}\,,
\end{equation}
where $C(\tildeQ,T)$ is monotone increasing in $T$ and
independent of $\epsilon$.
\end{prop}

\begin{proof}
We divide the proof into several steps. In the first steps we prove
local existence, then we address estimates uniformly in $\epsilon>0$ and
conclude global well-posedness.
While carrying out the program, we verify an energy dissipation law which
implies some lower-order a priori estimates for $(u,Q)$,
formulate a higher-order energy estimate, and finally
take advantage of a cancellation
of critical higher-order terms, an observation that was already used
in~\cite{PaicuZarnescuSIMA2011,PaicuZarnescuARMA2012} in the case when $\O=\mathbb{R}^d$. An application of
Gronwall's inequality concludes the proof.

\medskip

\noindent
\textit{Step 1: Local existence in $X_T$.}
We show the local existence via a fixed-point approach. In this step we prove existence on a 
time interval $(0,T)$, where $T$ may depend on $\epsilon$.
Fix any $T>0$. For any $f\in L^2(\ts{T};\R^d)$, we consider the linear system
\begin{equation}\label{fixed}
\begin{alignedat}{2}
\dt{u}-\div\barg{ \nu(\tildeQ)\Lsym{u}  }
+\nabla p+f
&\,=F\quad &&\text{ in }(0,T)\times \Omega\,,\\
\diverg  u&\,=0&&\text{ in }(0,T)\times \Omega\, ,\\
\dt{Q}+\varepsilon\Delta^2 Q-\Delta Q-S(\nabla u, \tildeQ)
&\,=G&&\text{ in }(0,T)\times \Omega\,,\\
u|_{\partial\O}&\,=0 &&\text{ on }(0,T)\times\partial\O\,,\\
 Q|_{\partial\O}=\Delta Q|_{\partial\O}&\,=0&&
\text{ on }(0,T)\times\p\O\,,\\
 (u,Q)|_{t=0}&\,=(u_0,Q_0)\quad && \text{ on }\Omega\,.
 \end{alignedat}
\end{equation}
Note that~\bref{fixed} corresponds to~\bref{approstrong} where the term
$-\diverg \sigma(\tildeQ, Q)$ has been replaced by $f$.
By Lemma~\ref{stok}, the first equation of~\eqref{fixed} has a
unique solution
$u\in X_{T;u}$ on $(0,T)$.
Then $S(\nabla u,\tildeQ)\in L^2(\ts{T};\mathbb{S}_0)$, recall~\eqref{tensors},
and by Lemma~\ref{para} the third equation of~\eqref{fixed}
has a unique solution $Q\in L^2(0,T;H^4(\O;\mathbb{S}_0))\cap H^1(0,T;L^2(\O;\mathbb{S}_0))$ on $(0,T)$.
As a result we obtain an affine mapping
\begin{align}\label{linear2}
\begin{split}
  \mathscr{F} \colon  L^2(\ts{T};\R^d)&\,\to L^2(\ts{T};\R^d)\,,\\
  f&\,\mapsto \mathscr{F}(f)=-\diverg  \sigma(\tildeQ,  Q)
=-\diverg(\tildeQ\Delta Q-\Delta Q\tildeQ)\,.
\end{split}
\end{align}
In order to employ Banach's fixed-point theorem to  $\mathscr{F}$,
we will show that
there exists a $T$ positive but small enough such that
$\mathscr{F}$ maps from $L^2(\ts{T};\R^d)$ into itself and
is a contraction.

\medskip

\noindent
\textit{Step 2: The map $\mathscr{F}$ is a contraction for $T>0$ small enough.}
Fix $T>0$, $f_1,f_2\in L^2(\O_T;\R^d)$ and denote by $(u_i, Q_i)$, $i=1,2$, the solutions
of the systems
\begin{equation}\label{contraction}
\begin{alignedat}{2}
\dt{u_{i}}-\div\barg{ \nu(\tildeQ)\Lsym{u_i}  }
+\nabla p_i+f_i
&\,=F\quad &&\text{ in }(0,T)\times \Omega\,,\\
\diverg  u_i&\,=0&&\text{ in }(0,T)\times \Omega\, ,\\
\dt{Q_i}+\varepsilon\Delta^2 Q_i-\Delta Q_i-S(\nabla u_i, \tildeQ)
&\,=G&&\text{ in }(0,T)\times \Omega\,,\\
u_i|_{\partial\Omega}&\,=0 &&\text{ on }(0,T)\times \partial\O\,,\\
 Q_i|_{\partial\Omega}=\Delta Q_i|_{\partial\Omega}&\,=0&&
\text{ on } (0,T)\times \p\O\,,\\
 (u_i,Q_i)|_{t=0}&\,=(u_0,Q_0) &\quad& \text{ in }\Omega\,.\\
 \end{alignedat}
\end{equation}
Let $\hat{u}=u_1-u_2$, $\hat{Q}=Q_1-Q_2$, $\hat{f}=f_1-f_2$,
and note that by~\eqref{linear2}
the identity
\begin{align*}
 \mathscr{F}(f_1)-\mathscr{F}(f_2)
=-\diverg (\tildeQ\Delta \hat{Q}-\Delta \hat{Q}\tildeQ)
\end{align*}
follows. Hence~\eqref{contraction} leads to
\begin{equation}\label{zero}
\begin{alignedat}{2}
\dt{\hat u}-\div\barg{ \nu(\tildeQ)D(\hat u) }
+\nabla \hat p+\hat f
&\,=0\quad &&\text{ in }(0,T)\times \Omega\,,\\
\diverg  \hat u&\,=0&&\text{ in }(0,T)\times \Omega\, ,\\
\dt{\hat Q}+\varepsilon\Delta^2 \hat Q-\Delta \hat Q-S(\nabla \hat u, \tildeQ)
&\,=0&&\text{ in }(0,T)\times \Omega\,,\\
\hat u|_{\partial\O}&\,=0 &&\text{ on }(0,T)\times\partial\O\,,\\
 \hat Q|_{\partial\O}=\Delta \hat Q|_{\partial\O}&\,=0&&
\text{ on }(0,T)\times \p\O\,,\\
 (\hat u,\hat Q)|_{t=0}&\,=(0,0) &\quad& \text{ in }\Omega\,.
 \end{alignedat}
\end{equation}
For almost every $t\in [0,T]$,
\begin{equation*}
  \|\mathscr{F}(f_1)(t)-\mathscr{F}(f_2)(t)\|_{L^2(\O)}
\leq \|\Delta \hat{Q}(t)\tildeQ-\tildeQ\Delta \hat{Q}(t)\|_{H^1(\O)}
\leq C\|\Delta \hat{Q}(t)\|_{H^1(\O)}\|\tildeQ\|_{H^2(\O)}\,.
\end{equation*}
 By the interpolation inequality $\|\hat{Q}\|_{H^3(\Omega)}\leq C\|\hat{Q}\|_{H^2(\Omega)}^{\frac12}\|\hat{Q}\|_{H^4(\Omega)}^{\frac12}$
\begin{equation*}
  \|\mathscr{F}(f_1)(t)-\mathscr{F}(f_2)(t)\|^2_{L^2(\O)}
\leq C\| \hat{Q}(t)\|_{H^2(\O)}\| \hat{Q}(t)\|_{H^4(\O)}\|\tildeQ\|^2_{H^2(\O)}\,.
\end{equation*}
As a result,
\begin{equation*} 
\begin{split}
\|\mathscr{F}(f_1)-\mathscr{F}(f_2)\|^2_{L^2(\ts{T})}
\leq & C\| \hat{Q}\|_{L^{\infty}(0,T;H^2)}
\|\tildeQ\|^2_{H^2(\O)} \|\hat{Q}\|_{L^1(0,T;H^4)}\\
\leq & C\| \hat{Q}\|_{L^{\infty}(0,T;H^2)}
\|\tildeQ\|^2_{H^2(\O)} T^{\frac 12}\|\hat{Q}\|_{L^2(0,T;H^4)}\,.
\end{split}
\end{equation*}
We estimate $\|\hat{Q}\|_{L^2(0,T;H^4)}$ and $\| \hat{Q}\|_{L^{\infty}(0,T;H^2)}$
by making use of the third equation in~\eqref{zero}.
In view of the estimates in Lemma~\ref{para} and \eqref{interpo2}
 \begin{equation*}
  \| \hat{Q}\|_{L^{\infty}(0,T;H^2)}^2+\|\hat{Q}\|^2_{L^2(0,T;H^4)}
\leq C(\varepsilon)\|S(\nabla \hat{u},\tildeQ)\|^2_{L^2(\ts{T})}\,,
\end{equation*}
where $C$ is independent of $T>0$.
 For $\|S(\nabla \hat{u},\tildeQ)\|_{L^2(\ts{T})}$
we use the first equation in~\eqref{zero} and the bound of $\hat{u}$ by
Lemma~\ref{stok},
\begin{equation} 
\begin{split}
  \|S(\nabla \hat{u},\tildeQ)\|_{L^2(\ts{T})}
\, \leq \|\tildeQ\|_{L^{\infty}(\O)}\|\hat{u}\|_{L^2(0,T;H^1)}&\leq C\|\tildeQ\|_{H^2(\O)}\|\hat{u}\|_{L^2(0,T;H^1)}\\
&\,  \leq C(\varepsilon)\|\tildeQ\|_{H^2(\O)}
   \|\hat{f}\|_{L^2(\ts{T})} \,.
  \end{split}
\end{equation}
Based on these estimates
we infer
 \begin{equation*}
\|\mathscr{F}(f_1)-\mathscr{F}(f_2)\|^2_{L^2(\ts{T})}
    \leq  T^{\frac 12}C(\varepsilon)\|\tildeQ\|^4_{H^2(\O)}
\|\hat{f}\|^2_{L^2(\ts{T})},
\end{equation*}
and consequently  $\mathscr{F}$ is a contraction provided that
$T \ll 1$. 
\medskip

\noindent
\textit{Step 3: The map $\mathscr{F}$  is a self-map.}
Employing the result in the previous step, we have
 \begin{equation*}
 \begin{split}
   \|\mathscr{F}(f)\|_{L^2(\ts{T})}&\leq
\|\mathscr{F}(f)-\mathscr{F}(0)\|^2_{L^2(\ts{T})}+
\|\mathscr{F}(0)\|^2_{L^2(\ts{T})}
    \\
    &\leq  T^{\frac 12}C(\varepsilon)\|\tildeQ\|^4_{H^2(\O)}
\|f\|^2_{L^2(\ts{T})}+\|\mathscr{F}(0)\|^2_{L^2(\ts{T})}.
 \end{split}
\end{equation*}
It remains to show that $\mathscr{F}(0)\in L^2(\O_T)$.
For almost every $t\in [0,T]$,
\begin{equation*}
  \|\mathscr{F}(0)(t)\|_{L^2(\O)}
\leq \|\Delta Q(t)\tildeQ-\tildeQ\Delta Q(t)\|_{H^1(\O)}
\leq C\|\Delta Q(t)\|_{H^1(\O)}\|\tildeQ\|_{H^2(\O)}\,.
\end{equation*}
  As a result
  \begin{equation}\label{fixpoint1}
    \|\mathscr{F}(0)(t)\|_{L^2(\O_T)}\leq 
  C\| Q(t)\|_{L^2(0,T;H^4(\O))}\|\tildeQ\|_{H^2(\O)}\,.
  \end{equation}
We estimate $\|Q\|_{L^2(0,T;H^4)}$
by making use of the third equation in~\eqref{fixed}.
In view of the estimates in Lemma~\ref{para},
\begin{equation}\label{fixpoint2}
  \|Q\|^2_{L^2(0,T;H^4)}
\leq C(\varepsilon)(\|Q_0\|^2_{H^2(\O)}+\|G\|^2_{L^2(\ts{T})}
+\|S(\nabla u,\tildeQ)\|^2_{L^2(\ts{T})})\,.
\end{equation}
To estimate $\|S(\nabla u,\tildeQ)\|_{L^2(\ts{T})}$
we use again the first equation in~\eqref{fixed} and the bound of $u$ by
Lemma~\ref{stok},
\begin{equation}\label{fixpoint3}
\begin{split}
  \|S(\nabla u,\tildeQ)\|_{L^2(\ts{T})}
&\, \leq \|\tildeQ\|_{L^{\infty}(\O)}\|u\|_{L^2(0,T;H^1)}\leq C\|\tildeQ\|_{H^2(\O)}\|u\|_{L^2(0,T;H^1)}\\
&\,  \leq C(\varepsilon)\|\tildeQ\|_{H^2(\O)}
  \barg{ \|u_0\|_{H^1_{0,\sigma}(\O)}+\|F\|_{L^2(\ts{T})} }\,.
  \end{split}
\end{equation}
We deduce from~\eqref{fixpoint1},~\eqref{fixpoint2} and~\eqref{fixpoint3} that:
\begin{equation}\label{contrac1}
  \|\mathscr{F}(0)\|^2_{L^2(\ts{T})}
\leq C(\varepsilon,u_0,\tildeQ,F,G)\,.
\end{equation}
Combining the assertion of Step 2 and 3 along with Banach's fixed-point theorem implies that there exists a unique solution
on the time interval $(0, T)$. 
\medskip

\noindent
\textit{Step 4: The basic energy estimate.}
In order to prove the existence of a solution $(u,Q)$ on $(0,T)$ for
$T>0$ given,
suppose that $(u,Q)$ is a solution on $(0,T)$ with the
regularity in the assertion of the proposition. If we establish the
a~priori estimate in the assertion of the proposition, then we obtain some $\Teps>0$ such that the solution is uniquely
determined on $(0,\Teps)$.
In view of the uniform bounds on $(0,T)$ we may
solve on $(\Teps, 2\Teps)$ and obtain the existence of a unique solution in
finitely many steps.

Here we establish a first estimate. Fix $T>0$.
There exists a constant $C$
which may depend on $T$ and may be monotone increasing in $T$ but is
independent of $\epsilon$ such that
\begin{align*}
   \|\nabla Q\|_{L^{\infty}(0,T;L^2)}
&+\|u\|_{L^{\infty}(0,T;L^2)}
+\|\nabla u\|_{L^2(\ts{T})}+\|\Delta Q\|_{L^2(\ts{T})}\\
 &  \leq C\barg{ \|F\|_{L^2(\ts{T})}+\|G\|_{L^2(\ts{T})}
+\|u_0\|_{L^2_{\sigma}(\O)}+\|\nabla Q_0\|_{L^2(\O)} }\,.
\end{align*}
The proof is based on the observation that,
like the fully nonlinear system~\eqref{strongformeqn},
the approximating linearized
system~\eqref{approstrong} satisfies an energy dissipation law. To obtain this law,
use $u$ and $\Delta Q$ as test functions in the equations for $u$ and $ Q$,
respectively, to obtain
\begin{equation*}
  \frac 12\frac{d}{dt}\int_{\O}|u|^2
+\int_{\O}\nu(\tildeQ)|\Lsym{u} |^2
-\int_{\O}\diverg  \sigma(\tildeQ,  Q)\cdot u  =\int_{\O}F\cdot u
\end{equation*}
and
\begin{equation*}
  \frac 12\frac{d}{dt}\int_{\O}|\nabla Q|^2
+\varepsilon\int_{\O}|\nabla\Delta Q|^2
+\int_{\O}|\Delta Q|^2+\int_{\O}S(\nabla u,\tildeQ) : \Delta Q
=-\int_{\O}G : \Delta Q\,
\end{equation*}
for every $t\in (0,T)$.
The key observation is that the algebraic identity~\eqref{tensor4}
leads together with an integration by parts to
\begin{equation*}\label{cancel}
  \begin{split}
    \int_{\O}S(\nabla u,\tildeQ) : \Delta Q
=\int_{\O}  \barg{(\Delta Q)\tildeQ-\tildeQ\Delta Q  }:\nabla u
=-\int_{\O}\sigma(\tildeQ,  Q):\nabla u
=\int_{\O}\diverg\sigma(\tildeQ,  Q)\cdot u\,.
  \end{split}
\end{equation*}
Therefore the critical terms cancel and the combination
of the foregoing identities implies
\begin{align*}
  \frac 12\frac{d}{dt}\int_{\O}\barg{ |\nabla Q|^2+|u|^2 }
+\int_{\O}\nu(\tildeQ)|\Lsym{u} |^2
+\varepsilon\int_{\O}|\nabla\Delta Q|^2
+\int_{\O}|\Delta Q|^2
  =\int_{\O}F\cdot u-\int_{\O}G:\Delta Q\,.
\end{align*}
By~\eqref{viscosity}, the Cauchy-Schwarz and Young inequality,
Gronwall's Lemma, and Korn's inequality we obtain the assertion in this step.
This estimate will be used frequently in the sequel
when we estimate the lower-order terms.

\medskip

\noindent
\textit{Step 5: Higher-order energy estimate.}
In this step we assert the validity of the estimate
\begin{align}\nonumber
&
\frac 12\frac{d}{dt}\int_{\O}\nu(\tildeQ)\barg{2|\Lsym{u} |^2+|\Delta Q|^2 }
 \\\label{total} &\qquad
+\int_{\O}\barg{
2|Au|^2+ \nu(\tildeQ)|\nabla \Delta Q|^2
    +\varepsilon \nu(\tildeQ)|\Delta^2  Q|^2 } + \mathcal{A}
= \sum_{i=1}^5\mathcal{J}_i\,,
\end{align}
where $\mathcal{A}$ and the terms $\mathcal{J}_i$ are defined below,
see~\bref{callow},~\bref{defA}, and~\bref{defJ16}.
From this estimate we infer the estimates asserted in the proposition.

Unless otherwise indicated, the calculations in this step
are carried out for almost every $t\in [0,T]$.
This is possible since we already showed the existence of a local solution
on $(0,\Teps)$
which has the regularity given by~\eqref{regularity}.
In the following we denote $\tilde{A} u=-\div(\nu(\tildeQ)\Lsym{u})$. Then
\begin{equation}\label{stokesoper}
  Au = P_\sigma\tilde{A} u.
\end{equation}
We test the equation for $u$ in~\eqref{approstrong} by $Au$
and get
\begin{equation*}
 \int_\Omega \left( \dt{u}\cdot Au + |Au|^2
-\diverg \sigma(\tildeQ, Q)\cdot Au \right)
  =\int_\Omega F\cdot Au\,.
\end{equation*}
Since $u$ is divergence-free and in view of the boundary conditions for $u$,
an integration by parts yields
\begin{align*}
\begin{split}
   \frac 12\frac{d}{dt}\int_{\O}\nu(\tildeQ)|\Lsym{u} |^2
+\int_{\O}|Au|^2   -\int_{\O}\diverg\sigma(\tildeQ,  Q)\cdot Au
-\int_{\O}F \cdot Au=0\,.
\end{split}
\end{align*}
Note that the  Helmholtz decomposition for $\tilde{A}u$ is given by 
\begin{align*}
 \tilde{A}u=Au+\nabla q\quad \text{ for some }q\in H^1(\O)\,,
\end{align*}
cf. \eqref{eq:Helmholtz}.
Hence
\begin{equation*}
  \int_{\O}\diverg\sigma(\tildeQ,  Q) \cdot \tilde{A}u
=\int_{\O}\diverg\sigma(\tildeQ,   Q) \cdot Au
+\int_{\O}\diverg\sigma(\tildeQ, Q)\cdot \nabla q\,.
\end{equation*}
Choose a sequence $(q_k)_{k\in \mathbb{N}}\subset H^2(\Omega)$ with $q_k\to_{k\to\infty} q$
in $H^1(\Omega)$. Since $\sigma$ is skew-symmetric and since
$\tildeQ$ vanishes on $\p\O$, we infer
\begin{equation*}
  \int_{\O}\diverg\sigma(\tildeQ,  Q)\cdot \nabla q
=-\lim_{k\to\infty}
\int_{\O}\sigma(\tildeQ,  Q): \nabla^2 q_k = 0\,.
\end{equation*}
The foregoing identities
imply the following energy-type estimate for $u$,
\begin{align}\label{velicitynew}
\begin{split}
   \frac 12\frac{d}{dt}\int_{\O}\nu(\tildeQ)|\Lsym{u} |^2
+\int_{\O}|Au|^2
   -\int_{\O}\diverg \sigma(\tildeQ, Q)\cdot \tilde{A}u
+\int_{\O}F \cdot Au=0\,.
\end{split}
\end{align}
Now we consider the equation for $Q$
and use $\Delta(\nu(\tildeQ)\Delta Q)$ as  a test function for the evolution equation for $Q$ to obtain
\begin{equation*}
\begin{split}
    &\int_\Omega \left(\dt{Q}: \Delta  \barg{ \nu(\tildeQ)\Delta Q}
+\varepsilon\Delta^2  Q:\Delta \barg{ \nu(\tildeQ)\Delta Q} \right)\\
&\qquad
+\int_\Omega
\left( -\Delta Q : \Delta \barg{ \nu(\tildeQ)\Delta Q }
-S(\nabla u,\tildeQ) : \Delta \barg{ \nu(\tildeQ)\Delta Q } \right)
 =\int_\Omega G : \Delta \barg{\nu(\tildeQ)\Delta Q }\,.
  \end{split}
\end{equation*}
In view of the boundary conditions for $Q$ we deduce
\begin{equation*}
  \begin{split}
   \int_{\O}\dt{Q} : \Delta \barg{\nu(\tildeQ)\Delta Q }
    =\frac 12\frac{d}{dt}\int_{\O}\nu(\tildeQ)|\Delta Q|^2
  \end{split}
\end{equation*}
and analogously with one integration by parts and product rule,
\begin{equation*}
  \begin{split}
    \int_{\O}-\Delta Q : \Delta \barg{\nu(\tildeQ)\Delta Q }
    =\int_{\O}\nu(\tildeQ)|\nabla \Delta Q|^2
+\int_{\O}\nabla \Delta Q : \nabla  (\nu(\tildeQ)) \otimes \Delta Q\,.
  \end{split}
\end{equation*}
Here we write for simplicity
$\nabla \Delta Q : \nabla  (\nu(\tildeQ)) \otimes \Delta Q
=\partial_\gamma\Delta Q_{\alpha\beta}\partial_\gamma (\nu(\tildeQ))
\Delta Q_{\alpha\beta}$.
The combination of these three identities for $Q$
yields:
\begin{equation*}
  \begin{split}
    \frac 12\frac{d}{dt}\int_{\O}\nu(\tildeQ)|\Delta Q|^2
+\int_{\O}\nu(\tildeQ)|\nabla \Delta Q|^2
+\int_{\O}\nabla \Delta Q : \nabla  (\nu(\tildeQ))\otimes \Delta Q\\
    +\varepsilon\int_{\O}\nu(\tildeQ)|\Delta^2  Q|^2
+2\varepsilon\int_{\O}
\Delta^2  Q_{\alpha\beta}
\barg{ \nabla ( \nu(\tildeQ)) \cdot\nabla\Delta Q_{\alpha\beta}}
+\varepsilon\int_{\O} \Delta  (\nu(\tildeQ))\Delta^2  Q: \Delta Q \\
    =\int_{\O}S(\nabla u,\tildeQ) : \Delta \barg{\nu(\tildeQ)\Delta Q }
+\int_{\O}G : \Delta \barg{ \nu(\tildeQ)\Delta Q }\,.
  \end{split}
\end{equation*}
We can rewrite the above equation by leaving all the
lower-order terms on the right-hand side,
\begin{equation}\label{qes3}
  \begin{split}
    \frac 12\frac{d}{dt}\int_{\O}\nu(\tildeQ)|\Delta Q|^2
&+\int_{\O}\nu(\tildeQ)|\nabla \Delta Q|^2
    +\varepsilon\int_{\O}\nu(\tildeQ)|\Delta^2  Q|  \\
  &\quad \quad  =\int_{\O}S(\nabla u,\tildeQ) : \Delta \barg{\nu(\tildeQ)\Delta Q }
+\mathcal{V}\,,
  \end{split}
\end{equation}
where $\mathcal{V}$  is defined by 
\begin{equation}\label{callow}
\begin{split}
    \mathcal{V}
=&
-\int_{\O}
\nabla \Delta Q:\barg{ \nabla ( \nu(\tildeQ))\otimes \Delta Q }
-2\varepsilon\int_{\O}\Delta^2  Q_{\alpha\beta}
\nabla  (\nu(\tildeQ))\cdot\nabla\Delta Q_{\alpha\beta} \\
&-\varepsilon\int_{\O}\Delta  (\nu(\tildeQ))
\Delta^2  Q:\Delta Q
+\int_{\O}G : \Delta \barg{ \nu(\tildeQ)\Delta Q }
\equiv\sum_{i=1}^4\mathcal{J}_i\,.
\end{split}
\end{equation}
Finally  the combination of~\eqref{velicitynew} and~\eqref{qes3} gives
the assertion of this step, where
\begin{align}\label{defA}
\mathcal{A} = -2\int_{\O}\diverg\sigma(\tildeQ,  Q)
\cdot \tilde{A}u
-\int_{\O}S(\nabla u,\tildeQ):\Delta \barg{ \nu(\tildeQ)\Delta Q }
\end{align}
and
\begin{align}\label{defJ16}
\mathcal{J}_5 = -2\int_{\O}F \cdot Au\,.
\end{align}

\medskip

\noindent
\textit{Step 6: Estimates for the terms $\mathcal{J}_i$:}
These estimates are routine calculations and deferred to
Appendix~\ref{appendixA}. There it is shown that for all
$\delta>0$ there exists a constant $C=C(\delta, \Omega, Q_0)$ such that
  \begin{align}\nonumber
         \sum_{i=1}^5\mathcal{J}_i
\leq &\, 2\delta\|\nabla\Delta Q\|^2_{L^2}
+C(\delta,\tildeQ)\barg{ \|F\|_{L^2}^2+\|G\|^2_{H^1}+\|\Delta Q\|^2_{L^2} }
+\frac{3\varepsilon}{5}\|\Delta^2Q\|^2_{L^2}\\\label{total1}
&     +3\varepsilon C(\tildeQ)\|\nabla\Delta Q\|^2_{L^2}
+\delta \|Au\|_{L^2}^2
\,.
\end{align}

\medskip

\noindent
\textit{Step 7: The cancellation property and the estimate for $\mathcal{A}$ .}
This is the most important part in the uniform estimate, which shows that the highest order terms in $\mathcal{A}$ have the cancellation property and the remaining parts can be controlled. In order to show this, we shall integrate by parts several times and third order derivatives of $u$, like $\Delta \nabla u$, might appear, while in the first step we only show  $H^2$-regularity for $u$. However, noticing that $\mathcal{A}$ depends on $u$ (along with its derivatives) linearly, by a standard density argument, we can assume that $\Delta\nabla u\in L^2(\O)$ for almost every $t$. The point is, in the final form, $\mathcal{A}$ only contains $u$ and its derivatives up to order 2.
We start by estimating  $\mathcal{A}$ defined in~\eqref{defA}.
By~\eqref{tensor4}
\begin{equation*}
\mathcal{A}
=-2\int_{\O}\diverg\sigma( \tildeQ,  Q)\cdot \tilde{A} u
+ \int_{\O}
\sigma\barg{ \tildeQ,{ \nu(\tildeQ )\Delta Q }  } :\nabla u
=\mathcal{A}_1 + \mathcal{A}_2 \,.
\end{equation*}
For $\mathcal{A}_1$,  we integrate by parts, employ the boundary
conditions for $Q$ in~\eqref{approstrong},
use $2\diverg\Lsym{u} = \Delta u$ and infer
\begin{equation*}
\begin{split}
   \mathcal{A}_1
&=-\int_{\O}\diverg\sigma( \tildeQ,  Q)\cdot 2\tilde{A} u
=\int_{\O}\diverg\sigma( \tildeQ,  Q) \cdot
\barg{ 2\Lsym{u} \nabla(\nu(\tildeQ)) + \nu(\tildeQ) \Delta u }\\
  &=\int_{\O}2\diverg\sigma( \tildeQ,  Q) \cdot
(\Lsym{u} \nabla(\nu(\tildeQ)))
-\int_{\O}\sigma(\tildeQ,  Q) :
\barg{
\Delta u \otimes \nabla(\nu(\tildeQ))}
\\ &\quad -\int_{\O}\sigma(\tildeQ,  Q) :
\barg{ \nu(\tildeQ) \Delta \nabla u} 
=\mathcal{I}_1+\mathcal{I}_2+\mathcal{I}_3\,.
\end{split}
\end{equation*}
For $\mathcal{A}_2$ we note that by~\bref{linalg} for all $A$, $B$, $C\in \R^{d\times d}$
with $A^T=A$ the identity $$(AB-BA):C = (AC-CA):B$$ holds.
We integrate by parts twice and discover
\begin{equation*}
  \begin{split}
     \mathcal{A}_2 &  =
 \int_{\O}\nabla u :
\barg{ \tildeQ\Delta
\barg{ \nu(\tildeQ )\Delta Q }- \Delta\barg{ \nu(\tildeQ)\Delta Q }
\tildeQ } \\
&= \int_{\O}\nu(\tildeQ )
\Delta (\tildeQ\nabla u) : \Delta Q
-\nu(\tildeQ) \Delta (\nabla u\tildeQ) : \Delta Q\\
&=\int_{\O} \nu(\tildeQ )
\barg{ \tildeQ\Delta \nabla u - \Delta \nabla u\tildeQ} : \Delta Q
+\int_{\O} \nu(\tildeQ )
\barg{ \Delta \tildeQ\nabla u - \nabla u\Delta \tildeQ}: \Delta Q \\
&\quad +2\int_{\O}\nu(\tildeQ)\nabla \partial_\beta u_{\alpha}
\cdot\barg{ \nabla(\tildeQ)_{\alpha\gamma}\Delta Q_{\gamma\beta}-
   \nabla (\tildeQ)_{\gamma\beta}  \Delta Q_{\alpha\gamma} }=\mathcal{I}_4
+\mathcal{I}_5+\mathcal{I}_6\,.
  \end{split}
\end{equation*}
Notice that $\mathcal{I}_3+\mathcal{I}_4=0$ and that we can rewrite
$\mathcal{A}$ as
\begin{align*}
\mathcal{A}
&  = \int_{\O}\diverg\sigma( \tildeQ, Q) \cdot
(2\Lsym{u} \nabla(\nu(\tildeQ)))
-\int_{\O}\sigma(\tildeQ, Q) :
\barg{ \Delta u \otimes \nabla(\nu(\tildeQ)) } \\
& + \int_{\O} \nu(\tildeQ )\nabla u :
\barg{\Delta \tildeQ\Delta Q -\Delta Q\Delta \tildeQ} \\
&+2\int_{\O}\nu(\tildeQ)\nabla \partial_\beta u_{\alpha}
\cdot\barg{ \nabla(\tildeQ)_{\alpha\gamma}\Delta Q_{\gamma\beta}-
   \nabla (\tildeQ)_{\gamma\beta}  \Delta Q_{\alpha\gamma} }
=\mathcal{I}_1+\mathcal{I}_2+\mathcal{I}_5+\mathcal{I}_6\,.
\end{align*}
We prove in Appendix~\ref{appendixB} that
 \begin{equation}\label{total2}
  \mathcal{A}\leq 8\delta
\barg{ \|Au\|_{L^2}^2+\|\nabla\Delta Q\|_{L^2} ^2 }
+C(\delta,\tildeQ)\barg{ \|\nabla u\|_{L^2}^2+\|\Delta Q\|_{L^2}^2 }\,.
\end{equation}

\medskip

\noindent
\textit{Step 8: Proof of~\bref{regularity}.}
We combine the higher order estimate \eqref{total} in Step~5 with  
\eqref{total1} and~\eqref{total2} and choose
$\delta$ independently of $\epsilon$ small enough to obtain
\begin{align*} 
&\,  \frac{d}{dt}\int_{\O}\nu(\tildeQ)\barg{2 |\Lsym{u} |^2+|\Delta Q|^2 }
+\int_{\O}|Au|^2
+\frac{3}{4}\,\int_{\O}\nu(\tildeQ)|\nabla \Delta Q|^2
+\frac{2}{5}\,\varepsilon\int_{\O}\nu(\tildeQ)|\Delta^2  Q|^2  \\
&\, \qquad \leq C(\tildeQ)
\barg{ \|F\|_{L^2}^2+\|G\|^2_{H^1}+\|\Delta Q\|^2_{L^2}+
  \barg{\|\nabla u\|_{L^2}^2+\|\Delta Q\|^2_{L^2} }
 }   +2\varepsilon C(\tildeQ)\|\nabla\Delta Q\|^2_{L^2}\,.
\end{align*}
Fix $\varepsilon_0>0$ small enough so that one can absorb the last term
on the right-hand side for all $\epsilon\in (0, \epsilon_0)$
 by the corresponding term on the left-hand side. Then
\begin{align*} 
&\,  \frac{d}{dt}\int_{\O}\nu(\tildeQ)\barg{ |\Lsym{u} |^2+|\Delta Q|^2 }
+\frac{1}{2C_{\mathcal{S}}}\, \int_{\O}|\nabla ^2u|^2
+\frac{1}{2}\,\int_{\O}|\nabla \Delta Q|^2\nu(\tildeQ)
    +\varepsilon\int_{\O}\nu(\tildeQ)|\Delta^2  Q|^2  \\
&\, \qquad   \leq C(\tildeQ)
\barg{ \|F\|_{L^2}^2+\|G\|^2_{H^1}
  +\|\Delta Q\|^2_{L^2}
  + \|\nabla u\|_{L^2}^2+\|\Delta Q\|^2_{L^2}  }\,.
\end{align*}
An application of Gronwall's
inequality 
establishes~\eqref{bound}. Note that the bound depends exponentially on the
norm of the initial data $(\tildeQ,u_0)$.
\end{proof}


\section{Well-Posedness for the Linearized System}
\label{linearsystem}
In this section we define a linearization of
the full system which leads to a bounded linear operator which is one-to-one
and onto between the function spaces defined in Section~\ref{prelim}.
For simplicity, we set $\lambda=\Gamma=a=b=c=1$ in~\eqref{strongformeqn}
and~\eqref{strongforcing}.

The linear operator $\mathcal{L}(\tildeQ):X_T\to Y_T\times Y_0$ is given by
\begin{equation}\label{linearization}
\begin{split}
  \mathcal{L}(\tildeQ)(u,Q)
=\left(
\begin{array}{cc}
\dt{u}-  P_{\sigma}\diverg \bsqb{ \barg{ \nu(\tildeQ)\Lsym{u} }
+ \sigma(\tildeQ,  Q)}\\[.1in]
\dt{Q}-\Delta Q-S(\nabla u, \tildeQ) \\[.1in]
(u,Q)|_{t=0}
\end{array}\right)\,.
\end{split}
\end{equation}
The main result in this section concerns the global existence and
boundedness of solutions.

\begin{prop}[Linearized system]\label{linearized}
The linear operator $\mathcal{L}$ defined by~\eqref{linearization}
is an isomorphism between the spaces $X_T$ and
$Y_T\times Y_0$ and
\begin{equation*}
 \|\mathcal{L}^{-1}(\tildeQ)\|_{\mathcal{L}(Y_T\times Y_0,X_T)}
\leq C_{\mathcal{L}}(\tildeQ)\,.
\end{equation*}
\end{prop}

\begin{proof} Fix $T>0$.
We need to show that for all $(F,G,u_0,Q_0)\in Y_T\times Y_0$ there
exists a unique pair $(u,Q)\in X_T$ such that
\begin{equation} \label{operatorequ}
   \mathcal{L}(\tildeQ)(u,Q)
= (F,G,u_0,Q_0)\in Y_T\times Y_0\,.
  \end{equation}
The idea is to invoke Proposition~\ref{approximate1},
to pass to the limit $\varepsilon\to 0$, and to make use
of the weak compactness given by the bounds which hold
uniformly in $\epsilon$.
More precisely, for each $\varepsilon\in (0, \epsilon_0)$
with $\epsilon_0$ as in Proposition~\ref{approximate1},
the system~\eqref{approstrong} has a unique solution
$(u_{\varepsilon},Q_{\varepsilon})$
%
%
%
and~\eqref{bound} implies the bounds
  \begin{align}\label{newbound}
\begin{split}
&   \|u_{\varepsilon}\|_{L^{\infty}(0,T;H^1)}
+\|Q_{\varepsilon}\|_{L^{\infty}(0,T;H^2)}
+\|u_{\varepsilon}\|_{L^2(0,T;H^2)}
+\|Q_{\varepsilon}\|_{L^2(0,T;H^3)}\\
&\qquad + \sqrt\varepsilon\|\Delta^2 Q_{\varepsilon}\|_{L^2(\O_T)}  \leq C( \tildeQ)\|(F,G,u_0,Q_0)\|_{Y_T\times Y_0}\,.
 \end{split}
\end{align}
 Moreover,  by~\eqref{important}, 
 we estimate
\begin{equation}\label{eq:SEstim}
\|S(\nabla u_{\varepsilon},\tildeQ)\|_{L^2(0,T;H^1)}\leq C(\tildeQ)\|u_{\varepsilon}\|_{L^2(0,T;H^2)}\,.
\end{equation}
Hence the equation for $Q_\varepsilon$ implies that $\dt{ Q_\varepsilon}$ is uniformly bounded in $L^2(0,T;L^2(\Omega;\mathbb{S}_0))$ since $\sqrt{\varepsilon}\Delta^2 Q$ is bounded in $L^2(0,T;L^2(\Omega;\mathbb{S}_0))$. Similarly, since the right-hand side of
 \begin{equation*}
 \dt{u}_\varepsilon-P_\sigma\div\barg{ \nu(\tildeQ)\Lsym{u_\varepsilon} }
 =P_\sigma\diverg(\tildeQ\Delta Q_\varepsilon-\Delta Q_\varepsilon\tildeQ)+P_\sigma F
 \end{equation*}
is bounded $L^2(0,T;\H)$, $\dt{u_\varepsilon}$ is bounded in $L^2(0,T;\H)$.

By weak compactness, there exists a subsequence $\epsilon_k\to 0$ for
$k\to\infty$, corresponding pairs
$(u_k, Q_k)=(u_{\epsilon_k},Q_{\epsilon_k})$ and an element $(u,Q)$  such that
 \begin{equation}\label{weakcon}
\begin{alignedat}{3}
u_{k} &\weakly u\quad && \text{in}\quad&& L^2(0,T;H^1_{0,\sigma}(\O)\cap H^2(\O;\R^d))\cap H^1(0,T;\H)\,,\\
u_{k} &\weaklystar u &&\text{in} &&  L^{\infty}(0,T;H^1(\O))\,,\\
Q_{k} &\weakly   Q &&\text{in} &&L^2(0,T;H^1_0(\O;\mathbb{S}_0)\cap H^3(\O;\mathbb{S}_0))\cap H^1(0,T;L^2(\O;\mathbb{S}_0))\,,\\
Q_{k} &\weaklystar   Q &&\text{in}&& L^{\infty}(0,T; H^2(\O;\mathbb{S}_0))\,,
\end{alignedat}
\end{equation}
where $\Delta Q|_{\partial \Omega} =0$ and $(u,Q)|_{t=0} = (u_0,Q_0)$ in $\H\times H^1(\Omega;\mathbb{S}_0)$ since the corresponding trace maps are linear and bounded.
By  weak sequential lower semi-continuity  of norms
it follows
from~\eqref{weakcon} and~\eqref{newbound} that
\begin{align}\label{newbound1}
\begin{split}
& \|u\|^2_{L^{\infty}(0,T;H^1(\O))}+\|Q\|^2_{L^{\infty}(0,T;H^2(\O))}
           +\|u\|^2_{L^2(0,T;H^2(\O))}+\|Q\|^2_{L^2(0,T;H^3(\O))}\\
&\qquad   \leq C( \tildeQ)\|(F,G,u_0,Q_0)\|^2_{Y_T\times Y_0}\,.
 \end{split}
\end{align}
It remains to improve the regularity of the time derivative of $Q$
along with the corresponding estimates
and to show that the weak limit $(u,Q)$ is the unique solution
of~\eqref{operatorequ}.
We first consider the third equation in~\eqref{approstrong},
 \begin{equation}\label{weakq}
\dt{ Q_{k}} +
\epsilon_k\Delta^2 Q_{k}
-\Delta Q_{k}-S(\nabla u_{k},\tildeQ)=G.
\end{equation}
In view of~\eqref{weakcon} and  $\epsilon_k\Delta^2Q_k\to 0 $ in $L^2(0,T;L^2(\Omega;\mathbb{S}_0))$ as $k\to\infty$, we may pass to the limit $k\to\infty$ in~\eqref{weakq}
and infer
\begin{equation}\label{weakq1}
   \partial_t Q\
=\Delta Q+S(\nabla u,\tildeQ)+G.
 \end{equation}
Since the right-hand side belongs to $ L^2(0,T;H^1_0(\mathbb{S}_0))$, we obtain
$\dt{Q}\in L^2(0,T;H^1_0(\mathbb{S}_0))$.
Moreover, we can use the equation and \eqref{eq:SEstim} to estimate $\dt{Q}$ and conclude
 \begin{equation}\label{timeder}
   \|Q\|_{X_{T;Q}}
   \leq C(\tildeQ)
\barg{ \|Q_0\|_{H^2}+\|u\|_{L^2(0,T;H^2)}+\|G\|_{L^2(0,T;H^1)} }.
\end{equation}
 Moreover, we may pass to the limit $k\to \infty$ in the first equation
 of~\eqref{approstrong} such that
 \begin{equation}\label{sigle}
 \dt{u}-P_\sigma\div\barg{ \nu(\tildeQ)\Lsym{u} }
 =P_\sigma\diverg(\tildeQ\Delta Q-\Delta Q\tildeQ)+P_\sigma F
 \end{equation}
holds in $L^2(0,T;L^2_\sigma(\Omega))$.
 Notice that the
 right-hand side of~\eqref{sigle} belongs to $L^2(0,T;\H)$.
Hence by Lemma~\ref{stok}
\begin{equation*}
    \|u\|_{X_{T;u}}\leq C( \tildeQ)
\barg{ \|u_0\|_{H^1_{0,\sigma}}
+\|\diverg(\tildeQ\Delta Q-\Delta Q\tildeQ)+F\|_{L^2(\ts{T})} }\,.
\end{equation*}
By~\eqref{important} 
one obtains 
\begin{equation*}
  \|\diverg(\tildeQ\Delta Q-\Delta Q\tildeQ)\|_{L^2(0,T;L^2_{\sigma})}
\leq C( \tildeQ)\|Q\|_{L^2(0,T;H^3)}\,.
\end{equation*}
The foregoing two estimates give
\begin{equation}\label{timeder2}
   \|u\|_{X_{T;u}}
    \leq C( \tildeQ)\barg{ \|u_0\|_{H^1_{0,\sigma}(\O)}+\|Q\|_{L^2(0,T;H^3)}+
    \|F\|_{L^2(\ts{T})} }.
\end{equation}
The combination of~\eqref{timeder},~\eqref{timeder2} and~\eqref{newbound1}
implies the stated boundedness of $\mathcal{L}^{-1}(Q_0)$.
\end{proof}


\section{Properties of the nonlinear operator}
\label{nonlinearoperator}
For a given pair $(u_0, Q_0)\in Y_0$
we define a nonlinear operator $\mathcal{N}:X_T\to Y_T\times Y_0$ by
\begin{equation}\label{nonlinear}
\begin{split}
  &\mathcal{N}(\tildeQ) (u,Q)=
 \\& \
\left(\begin{array}{cc}
P_{\sigma}\diverg
\left[ \barg{ (\nu(Q)-\nu(\tildeQ))\Lsym{u} } +
 \barg{ \tau(Q)+\sigma(Q-\tildeQ,  Q)-u\otimes u }\right]\\
-(u\cdot \nabla) Q +S(\nabla u,Q-\tildeQ) + L(Q)\\
(u_0,Q_0)
\end{array}\right)\,.
\end{split}
\end{equation}

The important properties of this nonlinear operator are formulated
 in the subsequent proposition. First we define a constant $M$ which shall be used later.
 For any $Q_0\in H^2(\O)$, let $\tilde{Q}(t,x)$ be the solution of the heat equation with initial data $Q_0$ and homogemenous Dirichlet boundary condition. Then we define $M>0$ to be such that
 \begin{equation}\label{heattro}
   \|(0,\tilde{Q})\|_{X_T}\leq M\|Q_0\|_{H^2(\O)}\, .
 \end{equation}

\begin{prop}
\label{nonlinearterms}
Suppose that $(u_0,Q_0)\in Y_0$, let $T\in (0,1]$ and
 $R>M\|Q_0\|_{H^2}$.
Then there is a constant
$C_{\mathcal{N}}(T,R)>0$ with the
following properties:
\begin{itemize}
 \item [(i)] $\mathcal{N}$ is bounded on $\BRXT$ subject to the given initial data, i.e., for all
$(u,Q)\in \BRXT$ with $Q|_{t=0}=Q_0$ one has the estimate
\begin{equation*}
\|\mathcal{N}(\tildeQ)(u,Q)\|_{Y_T\times Y_0}
\leq C_{\mathcal{N}}(T,R)\(R+M\|Q_0\|_{X_T}\)+ C\|(u_0,Q_0)\|_{Y_0}\,;
\end{equation*}
 \item [(ii)] $\mathcal{N}$ is Lipschitz continuous on $\BRXT$
 subject to the given initial data,
i.e., for
all $(u_1, Q_1)$, $(u_2, Q_2)\in \BRXT$ with $Q_i|_{t=0}=Q_0$, $i=1,2$,
one has the estimate
\begin{equation*}
\|\mathcal{N}(\tildeQ)(u_1,Q_1)-\mathcal{N}(\tildeQ)(u_2,Q_2)\|_{Y_T\times Y_0}
\leq C_{\mathcal{N}}(T,R)\|(u_1-u_2,Q_1-Q_2)\|_{X_T}\,;
\end{equation*}
 \item [(iii)] for all $R>0$ the map
$C_{\mathcal{N}}(\cdot,R):(0,\infty)\to [0,\infty)$
satisfies
\begin{align*}
 \lim_{T\to 0}C_{\mathcal{N}}(T,R)= 0\,.
\end{align*}
\end{itemize}
\end{prop}

\begin{proof}
We begin with the proof of the Lipschitz continuity
since it is the most demanding part of the proof. From this we easily get the
boundedness and the properties of the constant in the a~priori estimates.

\medskip

\noindent
\textit{Proof of (ii).} We divide the proof into several steps. In the
first step, we deduce estimates on the distance between various trajectories
of director fields. In the subsequent steps, we estimate all the terms in the
expression $\mathcal{N}(\tildeQ)(u_1,Q_1)-\mathcal{N}(\tildeQ)(u_2,Q_2)$
which need to be controlled for the Lipschitz continuity.

\medskip

\noindent
\textit{Step 1: Estimates for trajectories.}
Let $(\overline{u},\overline{Q})=(u_1-u_2,Q_1-Q_2)$.
By the fundamental theorem of calculus
  \begin{equation*}
    Q_1(t,x)-\tildeQ(x)
=\int^t_0 \barg{ \dt{Q_1}(\tau,x)}
\dv{\tau} + \barg{ Q_1(0,x) - \tildeQ(x) }\quad \text{for } t\in (0,T).
  \end{equation*}
We infer by H\"{o}lder's inequality that
\begin{align*}
  \|Q_1(t,\cdot)-\tildeQ\|_{H^1(\O)}
  &\, \leq \Barg{
    \int_\Omega t \int_0^t \bbar{ \dt{Q_{1}}(\tau,x)} ^2
\dv{\tau}\dv{x} }^{1/2}
%
 \leq  t^{\frac 12}\|\dt{Q_{1}}\|_{L^2(0,t;H^1)}
\,.
\end{align*}
The assumption $Q_1\in \BRXT$ 
implies
\begin{align}\label{f1}
\begin{split}
\|Q_1-\tildeQ\|_{L^{\infty}(0,T;H^1)}
 &\, \leq   T^{\frac 12}(\|\dt{Q_{1}}\|_{L^2(0,T;H^1)}
 \leq T^{\frac 12}R
 \end{split}
\end{align}
One obtains analogously that
     \begin{equation}\label{f2}
    \|Q_2-\tildeQ\|_{L^{\infty}(0,T;H^1)}
           \leq   T^{\frac 12}R
    \end{equation}
and
    \begin{equation}\label{f11}
  \|\overline{Q}\|_{L^{\infty}(0,T;H^1)}
= \|Q_1-Q_2\|_{L^{\infty}(0,T;H^1)}
\leq   T^{\frac 12}\|\dt{\overline{Q}}\|_{L^2(0,T;H^1)}
\leq T^{\frac 12}\|\overline{Q}\|_{X_{T;Q}}\,.
    \end{equation}
Moreover, we note that \eqref{interpo2} implies
    \begin{equation*}
  \|\overline{Q}\|_{L^{\infty}(0,T;H^2)}
\leq C \|\overline{Q}\|_{X_{T;Q}}\,
    \end{equation*}
with a constant $C>0$ independent of $T>0$.
This estimate will be used frequently in the following.

\medskip

\noindent
\textit{Step 2: Estimates for viscous stresses.}
The following estimate holds for suitable constants,
\begin{align*}
&\bnorm{ P_{\sigma}
\barg{ \diverg\barg{ (\nu(Q_1)-\nu(\tildeQ))\Lsym{u_1} }
-\diverg\barg{ (\nu(Q_2)-\nu(\tildeQ))\Lsym{u_2} } } }_{Y_{T;u}} \\
&\,\qquad
\leq C(R)T^{\frac 14}\left(\|\overline{Q}\| _{X_{T;Q}}
+\|  \overline{u}\|_{X_{T;u}}\right)
+
C(R)T^{\frac 12}\|(\overline{u},\overline{Q})\|_{X_T}\,.
\end{align*}
For the proof of this estimate,
 note that $P_\sigma$ does not increase the $L^2$-norm and that
for $i=1,2$,
\begin{equation*}
  \div\barg{ (\nu(Q_i)-\nu(\tildeQ))\Lsym{u_i} }
=\frac12\barg{ \nu(Q_i)-\nu(\tildeQ) }\Delta u_i
+\nabla\barg{ \nu(Q_i)-\nu(\tildeQ) }\Lsym{u_i} \,.
\end{equation*}
In view of this decomposition it suffices to prove the estimates
\begin{equation}\label{g4}
   \bnorm{ 
(\nu(Q_1)-\nu(\tildeQ))\Delta u_1
-(\nu(Q_2)-\nu(\tildeQ))\Delta u_2 
}_{L^2(\ts{T})}
\leq C(R)T^{\frac 14}\|(\overline{u},\overline{Q})\|_{X_{T}}
\end{equation}
and
\begin{equation}\label{part0}
  \bnorm{ \nabla\barg{ \nu(Q_1)-\nu(\tildeQ) }\Lsym{u_1}
-\nabla\barg{ \nu(Q_2)-\nu(\tildeQ) }\Lsym{u_2}  }_{L^2(\ts{T})}
\leq C(R)T^{\frac 12}\|(\overline{u},\overline{Q})\|_{X_T}\,.
\end{equation}
To prove~\bref{g4}, note that
\begin{equation*}
  \begin{split}
   &\| (\nu(Q_1)-\nu(\tildeQ))\Delta u_1
-(\nu(Q_2)-\nu(\tildeQ))\Delta u_2\|_{L^2(\ts{T})}\\
   &\quad = \|(\nu(Q_1)-\nu(Q_2))\Delta u_1
+(\nu(Q_2)-\nu(\tildeQ))\Delta\overline{u}\|_{L^2(\ts{T})}\\
&\quad \leq \|\nu(Q_1)-\nu(Q_2)\|_{L^{\infty}(\ts{T})}\|\Delta u_1\|_{L^2(\ts{T})}
+\|(\nu(Q_2)-\nu(\tildeQ))\|_{L^{\infty}(\ts{T})}\|\Delta \overline{u}\|_{L^2(\ts{T})}\\
&\quad \leq \|\nu(Q_1)-\nu(Q_2)\|_{L^{\infty}(\ts{T})}R+
\|\nu(Q_2)-\nu(\tildeQ)\|_{L^{\infty}(\ts{T})}\|  \overline{u}\|_{X_{T;u}}\,.
  \end{split}
\end{equation*}
In the last step, we used $(u_i,Q_i)\in \BRXT$, $i=1,2$. This fact also implies
\begin{equation}\label{infty}
 \|u_i\|_{C([0,T];H^1)}
+ \|Q_i\|_{C([0,T];H^2)}
+\|Q_i\|_{L^{\infty}(\ts{T})}\leq C(R)\,,\quad i=1,\,2\,.
\end{equation}
With~\eqref{boundvis} and~\eqref{f11} we conclude
\begin{align*}
\begin{split}
  \|\nu(Q_1)-\nu(Q_2)\|_{L^{\infty}(\ts{T})}
&\,\leq C(R)\|Q_1-Q_2\|_{L^{\infty}(\ts{T})}
\leq C(R)\|\overline{Q}\|_{L^{\infty}(0,T;H^1)}^{\frac 12}
\|\overline{Q}\|_{L^{\infty}(0,T;H^2)}^{\frac 12} \\
&\,
\leq C(R)T^{\frac 14}\|\overline{Q}\|^{\frac 12}_{X_{T;Q}}
\|\overline{Q}\|_{X_{T;Q}}^{\frac 12}
\leq  C(R)T^{\frac 14}\|\overline{Q}\| _{X_{T;Q}}\, .
\end{split}
\end{align*}
Based on~\eqref{f2},~\eqref{infty} and~\eqref{boundvis},
we can estimate $\nu(Q_2)-\nu(\tildeQ)$
in a similar way,
\begin{align*}\label{g2}
   \|\nu(Q_2)-\nu(\tildeQ)\|_{L^{\infty}(\ts{T})}
   & \leq C(R)\|Q_2-\tildeQ\|_{L^{\infty}(0,T;H^1)}^{\frac 12}
 \|Q_2-\tildeQ\|_{L^{\infty}(0,T;H^2)}^{\frac 12}\\
   &\leq C(R)T^{\frac 14}\,.
\end{align*}
With these estimate we obtain immediately~\bref{g4}.

To verify the estimate~\bref{part0}, we note that
\begin{equation*}
  \nu(Q_1)-\nu(Q_2)
=\int^1_0\frac{\rm d}{\rm d\tau}\nu\barg{ \tau Q_1 + (1-\tau) Q_2}\dv{\tau}
=\int^1_0 D \nu\barg{\tau Q_1 + (1-\tau) Q_2 }( \bar{Q})\dv{\tau}\,.
\end{equation*}
As a result,
\begin{align}\label{vis2}
\begin{split}
\frac{\partial}{\partial x_i}  \barg{ \nu(Q_1)-\nu(Q_2) }
&=\int^1_0  \nabla^2\nu
\barg{\tau Q_1 + (1-\tau) Q_2 }\left(\bar{Q} ,
\frac{\partial}{\partial x_i}\barg{\tau Q_1 + (1-\tau) Q_2 }\right) \dv{\tau}\\
&\quad
+\int^1_0 \nabla \nu\barg{\tau Q_1 + (1-\tau) Q_2 }\left(
\frac{\partial}{\partial x_i}\bar{Q}\right)
\dv{\tau}\,.
\end{split}
\end{align}
Now we combine~\eqref{boundvis},~\eqref{infty} and~\eqref{vis2}, and deduce,
for any $p>1$,
\begin{equation}\label{com}
  \begin{split}
    \|\nabla\barg{ \nu(Q_1)-\nu(Q_2) }\|_{L^p(\O)}
&\leq C(R)\bnorm{\, (|\nabla\bar{Q}|+|\nabla Q_2|)\, |\bar{Q}| \,}_{L^p(\O)}
+C(R)\|\nabla\bar{Q}\|_{L^p(\O)}\\
    &\leq C(R)\barg{ \|\bar{Q}\|_{L^{\infty}}+1 }
\|\nabla\bar{Q}\|_{L^p(\O)}
+ C(R)\|\,|\nabla Q_2|\, |\bar{Q}| \,\|_{L^p(\O)}\,.
  \end{split}
\end{equation}
Using~\eqref{com} with $p=6$, H\"{o}lder's inequality, Sobolev embedding,
\eqref{infty} and~\eqref{inter3}, we can estimate by
\begin{equation*}
 \begin{split}
    &\bnorm{ \nabla\barg{ \nu(Q_1)-\nu(\tildeQ) }\Lsym{u_1}
-\nabla\barg{ \nu(Q_2)-\nu(\tildeQ) }\Lsym{u_2} }_{L^2(\O)}\\
    &\quad
\leq \|\nabla\barg{ \nu(Q_1)-\nu(Q_2) }\Lsym{u_1} \|_{L^2(\O)}
+\|\nabla(\nu(Q_2))\Lsym{\bar{u}}\|_{L^2(\O)}
+\|\nabla(\nu(\tildeQ))\Lsym{\bar{u}}\|_{L^2(\O)}\\
    &\quad
 \leq C(R)\|\nabla\bar{Q}\|_{L^6(\O)}\|\Lsym{u_1} \|_{L^3(\O)}
+C(R) \|\,|\nabla Q_2|\, |\bar{Q}| \,\|_{L^6(\O)}\|\Lsym{u_1} \|_{L^3(\O)}\\
&\,\qquad +C(R)\barg{ \|\nabla Q_2\|_{L^6(\O)}\|\Lsym{\bar{u}}\|_{L^3(\O)}
+\|\Lsym{\bar{u}}\|_{L^2(\O)} }\\
    &\quad
 \leq C(R)\|\bar{Q}\|_{H^2(\O)}
\|\Lsym{u_1} \|^{\frac 12}_{L^2(\O)}\|\Lsym{u_1} \|^{\frac 12}_{H^1(\O)}
+C(R) \|\nabla Q_2\|_{L^6(\O)} \|  \bar{Q} \|_{L^\infty(\O)}
\|\Lsym{u_1} \|_{L^3(\O)}\\
&\,\qquad +C(R)\barg{ \|Q_2\|_{H^2(\O)}+1 }\|\Lsym{\bar{u}}\|^{\frac 12}_{L^2(\O)}
\|\Lsym{\bar{u}}\|^{\frac 12}_{H^1(\O)}\\
    &\quad
 \leq C(R)\|\bar{Q}\|_{L^\infty(0,T;H^2(\O))}\|\Lsym{u_1} \|^{\frac 12}_{H^1(\O)}
+C(R)\|\Lsym{\bar{u}}\|^{\frac 12}_{H^1(\O)}\\
&\,\qquad +C(R) \| \bar{Q} \|_{X_{T;Q}}
\|\Lsym{u_1} \|^{\frac 12}_{L^2(\O)}\|\Lsym{u_1} \|^{\frac 12}_{H^1(\O)}
\,.     \end{split}
\end{equation*}
Integration in time implies~\bref{part0}.

\medskip

\noindent
\textit{Step 3: The extra stress in the flow equation.} We have
\begin{align*}
&\|P_\sigma\barg{
\diverg\sigma(Q_1-\tildeQ,  Q_1)
-\diverg\sigma(Q_2-\tildeQ,  Q_2) } \|_{L^2(\O_T)}
 \leq CT^{\frac 14} (R+\|\tildeQ\|_{H^2}) \|\overline{Q}\|_{X_{T;Q}}\,.
\end{align*}
In fact, by~\eqref{important}: 
  \begin{equation}\label{long}
    \begin{split}
    &\|\diverg\sigma(Q_1-\tildeQ, Q_1)
-\diverg\sigma(Q_2-\tildeQ, Q_2)\|_{L^2}\\
      =&\|\diverg \big(\Delta Q_1(Q_1-\tildeQ)
-(Q_1-\tildeQ)\Delta Q_1\big)-
\diverg \big(\Delta Q_2(Q_2-\tildeQ)
-(Q_2-\tildeQ)\Delta Q_2\big)\|_{L^2}\\
      \leq&\|\Delta \overline{Q} (Q_1-\tildeQ)
+\Delta Q_2\overline{Q}-(Q_1-\tildeQ)\Delta \overline{Q}-
\overline{Q}\Delta Q_2\|_{H^1}\\
      \leq & C\|\Delta\overline{Q}\|_{H^1}
\|Q_1-\tildeQ\|_{H^1}^{\frac 12}\|Q_1-\tildeQ\|_{H^2}^{\frac 12}
+C\|\Delta Q_2\|_{H^1}\|\overline{Q}\|_{H^1}^{\frac 12}
\|\overline{Q}\|_{H^2}^{\frac 12}\,.
    \end{split}
  \end{equation}
  As a result:
    \begin{equation}\label{high}
    \begin{split}
      &\|\diverg \sigma(Q_1-\tildeQ, Q_2)-
\diverg \sigma(Q_2-\tildeQ, Q_2)\|_{L^2(\ts{T})}\\
\leq & C\|Q_1-\tildeQ\|_{L^{\infty}(0,T;H^1)}^{\frac 12}
\|Q_1-\tildeQ\|_{L^{\infty}(0,T;H^2)}^{\frac 12}
\|\overline{Q}\|_{L^2(0,T;H^3)}      \\
&+C\|\overline{Q}\|_{L^{\infty}(0,T;H^1)}^{\frac 12}
\|\overline{Q}\|_{L^{\infty}(0,T;H^2)}^{\frac 12}\|  Q_2\|_{L^2(0,T;H^3)}\,.
    \end{split}
  \end{equation}
Combining~\eqref{f1},~\eqref{f11} with~\eqref{high} gives:
\begin{equation}\label{part1}
\begin{split}
      &\|\diverg\sigma(Q_1-\tildeQ, Q_1)
-\diverg\sigma(Q_2-\tildeQ,Q_2)\|_{L^2(\ts{T})}
\leq   CT^{\frac 14} R \|\overline{Q}\|_{X_{T;Q}}\,,
\end{split}
\end{equation}
i.e., the assertion.

\medskip

\noindent
\textit{Step 4: The convection term in the Navier-Stokes equation.} We have
\begin{align*}
\|P_\sigma\barg{ u_1\cdot\nabla u_1-u_2\cdot\nabla u_2 } \|_{L^2(\O_T)}
\leq T^{\frac 18}C(R)\|\overline{u}\|_{X_{T;u}}\,.
\end{align*}
Indeed,
  \begin{equation*}
  \begin{split}
     & \|u_1\cdot\nabla u_1-u_2\cdot\nabla u_2\|_{L^2(\O)}\\
     \leq &C\|\overline{u}\|_{L^2(\O)}^{\frac 14}
\|\overline{u}\|_{H^1(\O)}^{\frac 34}
C\|u_1\|_{H^1(\O)}^{\frac 14}\|u_1\|_{H^2(\O)}^{\frac 34}
+C\|u_2\|_{L^2(\O)}^{\frac 14}\|u_2\|_{H^1(\O)}^{\frac 34}
C\|\overline{u}\|_{H^1(\O)}^{\frac 14}
\|\overline{u}\|_{H^2(\O)}^{\frac 34}\,.
  \end{split}
  \end{equation*}
  Therefore
    \begin{equation}\label{convection}
  \begin{split}
     & \|u_1\cdot\nabla u_1-u_2\cdot\nabla u_2\|_{L^2(\ts{T})}\\
     \leq & C(R)\|\overline{u}\|_{X_{T;u}}
\bnorm{ \|u_1\|_{H^2(\O)}^{\frac 34}(t) }_{L^2(0,T)}
+C(R)\|\overline{u}\|_{X_{T;u}}^{\frac 14}
\bnorm{ \|\overline{u}\|_{H^2(\O)}^{\frac 34}(t) }_{L^2(0,T)}\\
\leq & T^{\frac 18}C(R)\|\overline{u}\|_{X_{T;u}}\,.
  \end{split}
  \end{equation}

\medskip

\noindent
\textit{Step 5: The additional forces in the evolution of the director.} We have
\begin{align*}
    \|S(\nabla u_1,Q_1-\tildeQ)-S(\nabla u_2,Q_2-\tildeQ)\|_{L^2(0,T;H^1)}
\leq CT^{\frac 14}R\|(\overline{u},\overline{Q})\|_{X_T}\,.
\end{align*}
Note that
\begin{align*}
 S(\nabla u,Q-\tildeQ)
=\Lskew{u}(Q-\tildeQ)
-  (Q-\tildeQ)\Lskew{u}.
\end{align*}
By a similar argument as for~\eqref{long}, one obtains:
\begin{align*}
&   \|S(\nabla u_1,Q_1-\tildeQ)-S(\nabla u_2,Q_2-\tildeQ)\|_{H^1}\\
 &\quad \leq  \|\nabla u_1\|_{H^1}\|\overline{Q}\|_{H^1}^{\frac 12}
\|\overline{Q}\|^{\frac 12}_{H^2}
+\|\nabla \overline{u}\|_{H^1}\|Q_2-\tildeQ\|_{H^1}^{\frac 12}
\|Q_2-\tildeQ\|^{\frac 12}_{H^2}\,.
\end{align*}
The above estimate implies:
\begin{equation}\label{non1}
\begin{split}
&\|S(\nabla u_1,Q_1-\tildeQ)-S(\nabla u_2,Q_2-\tildeQ)\|_{L^2(0,T;H^1)}\\
&\leq 
\| u_1\|_{L^2(0,T;H^2)}\|\overline{Q}\|_{L^{\infty}(0,T;H^1)}^{\frac 12}
\|\overline{Q}\|^{\frac 12}_{L^{\infty}(0,T;H^2)}\\
&\quad +
\| \overline{u}\|_{L^2(0,T;H^2)}
\|Q_2-\tildeQ\|^{\frac 12}_{L^{\infty}(0,T;H^1)}
\|Q_2-\tildeQ\|^{\frac 12}_{L^{\infty}(0,T;H^2)}\,.
\end{split}
\end{equation}
Combining~\eqref{non1} with~\eqref{f2} and~\eqref{f11} implies
\begin{equation}\label{part2}
\begin{split}
&
\|S(\nabla u_1,Q_1-\tildeQ)-S(\nabla u_2,Q_2-\tildeQ)\|_{L^2(0,T;H^1)}\\
&\leq  
C(R)T^{\frac 14}\|\overline{Q}\|_{X_{T;Q}}^{\frac 12}
\|\overline{Q}\|_{X_{T;Q}}^{\frac 12}
+C\| \overline{u}\|_{X_{T;u}}T^{\frac 14}
R
\leq
C(R)T^{\frac 14}\|(\overline{u},\overline{Q})\|_{X_T}\,.
\end{split}
\end{equation}

\medskip

\noindent
\textit{Step 6: The Ericksen stress tensor.} We have
\begin{align*}
 \|P_\sigma\barg{ \diverg\tau(Q_1)-\diverg\tau(Q_2) } \|_{L^2(\O_T)}
\leq  C(R)T^{\frac 14}\| \overline{Q}\|_{X_{T;Q}}\,.
\end{align*}
With~\eqref{important} we find
\begin{align*}
\MoveEqLeft{\|\diverg\tau(Q_1)-\diverg\tau(Q_2)\|_{L^2(\O)}}\\
&\leq  \| \tau(Q_1)- \tau(Q_2)\|_{H^1(\O)}
=  \| \nabla\overline{Q}\odot (\nabla Q_1+\nabla Q_2)\|_{H^1(\O)}
\\
&\leq  C\| \nabla\overline{Q}\|_{H^1(\O)}^{\frac 12}
\| \nabla\overline{Q}\|_{H^2(\O)}^{\frac 12}
\left(\| \nabla Q_1\|_{H^1(\O)}+\|\nabla Q_2\|_{H^1(\O)}\right)\,.
\end{align*}
So we deduce from the foregoing estimate that:
\begin{equation}\label{part3}
\begin{split}
& \|\diverg\tau(Q_1)-\diverg\tau(Q_2)\|_{L^2(\ts{T})}\\
\leq &   C\| \overline{Q}\|^{\frac 12}_{L^{\infty}(0,T;H^2(\O))}
\left(\|  Q_1\|_{L^{\infty}(0,T;H^2(\O))}
+\| Q_2\|_{L^{\infty}(0,T;H^2(\O))}\right)
\bigg\|\|\overline{Q}\|_{H^3}^{\frac 12}(t)\bigg\|_{L^2(0,T)}\\
\leq  & C(R)\| \overline{Q}\|^{\frac 12}_{X_{T;Q}}T^{\frac 14}
\|\overline{Q}\|_{L^2(0,T;H^3)}^{\frac 12}
\leq     C(R)T^{\frac 14}\| \overline{Q}\|_{X_{T;Q}}\,.
\end{split}
\end{equation}

\medskip

\noindent
\textit{Step 7: The convection term in the equation for the $Q$-tensor.} We have
\begin{align*}
 \|u_1\cdot\nabla Q_1-u_2\cdot\nabla Q_2\|_{H^1(\O)}
\leq C(R)T^{\frac 14}\|(\overline{u},\overline{Q})\|_{X_T}\,.
\end{align*}
In view of~\eqref{important} we find
\begin{multline}
  \|u_1\cdot\nabla Q_1-u_2\cdot\nabla Q_2\|_{H^1(\O)}\\
  \leq C\|\overline{u}\|_{H^1(\O)}\|\nabla Q_1\|^{\frac 12}_{H^1(\O)}
\|\nabla Q_1\|_{H^2(\O)}^{\frac 12}
+C\|u_2\|_{H^1(\O)}^{\frac 12}\|u_2\|_{H^2(\O)}^{\frac 12}
\|\nabla \overline{Q}\|_{H^1(\O)}\,.
\end{multline}
The above estimates imply:
\begin{equation}\label{part4}
\begin{split}
&\|u_1\cdot\nabla Q_1-u_2\cdot\nabla Q_2\|_{L^2(0,T;H^1)}\\
\leq &
C\|\overline{u}\|_{L^{\infty}(0,T;H^1)}
\| Q_1\|^{\frac 12}_{L^{\infty}(0,T;H^2)}
\bigg\|\|\nabla Q_1(t)\|_{H^2}^{\frac 12}\bigg\|_{L^2(0,T)}\\
&
+C\|u_2\|_{L^{\infty}(0,T;H^1)}^{\frac 12}
\bigg\|\|u_2(t)\|_{H^2}^{\frac 12}\bigg\|_{L^2(0,T)}
\| \overline{Q}\|_{L^{\infty}(0,T;H^1)}\\
\leq &
C(R)\|\overline{u}\|_{X_{T;u}}T^{\frac 14}
\|  Q_1\|_{L^2(0,T;H^3)}^{\frac 12}
+C(R)T^{\frac 14}\|u_2\|_{L^2(0,T;H^2)}^{\frac 12}
\| \overline{Q}\|_{X_{T;Q}}\\
\leq & C(R)T^{\frac 14}\|(\overline{u},\overline{Q})\|_{X_T}\,.
\end{split}
\end{equation}

\medskip

\noindent
\textit{Step 8: The lower-order terms.} One can show
\begin{equation}\label{part5}
  \|L(Q_1)-L(Q_2)\|_{L^2(0,T;H^1)}\leq T^{\frac12}\|L(Q_1)-L(Q_2)\|_{L^\infty(0,T;H^1)}
\leq T^{\frac 12}C(R)\|\overline{Q}\|_{X_{T;Q}}
\end{equation}
in a straight forward manner.
Combining~\eqref{g4},~\eqref{part0},~\eqref{part1},~\eqref{convection},
\eqref{part2},~\eqref{part3},~\eqref{part4} and~\eqref{part5}
implies the Lipschitz continuity, as asserted.

\emph{Proof (i) and (iii):} From the estimates above one easily verifies (iii).
Finally, let $\tilde{Q}(t,x)$ be the trajectory defined by \eqref{heattro}. Then it is clear that $(u,Q)$ and $(0,\tilde{Q})$ satisfy the conditions for (ii) and we obtain
\begin{equation*}
  \begin{split}
    &\|\mathcal{N}(\tildeQ)(u,Q)\|_{Y_T\times Y_0}\\
&\leq  \|\mathcal{N}(\tildeQ)(u,Q)-\mathcal{N}(\tildeQ)(0,\tilde{Q})\|_{Y_T\times Y_0}+\|\mathcal{N}(\tildeQ)(0,\tilde{Q})\|_{Y_T\times Y_0}\\
&\leq C_{\mathcal{N}}(T,R)\|(u,Q-\tilde{Q})\|_{X_T}+ C\|(u_0,Q_0)\|_{Y_0}\\
&\leq C_{\mathcal{N}}(T,R)\(R+M\|Q_0\|_{X_T}\)+ C\|(u_0,Q_0)\|_{Y_0}.
  \end{split}
\end{equation*}

\end{proof}

\section{Proof of Theorem~\ref{maintheorem}}
\label{maintheoremproof}

Based on the properties of the linear operator in  Proposition~\ref{linearized}
and of the nonlinear operator in~Proposition~\ref{nonlinearterms},
we can give the proof of Theorem~\ref{maintheorem} which we divide into several steps.

\medskip

\noindent
\textit{Step 1: Reformulation as a fixed-point theorem.}
Note that $(u,Q)$ is a strong solution
of the system with given initial and boundary conditions if and only if 
\begin{equation*}
  \mathcal{L}(\tildeQ) (u,Q)
=\mathcal{N}(\tildeQ) (u,Q)\quad \text{ with }(u,Q)\in X_T\,.
\end{equation*}
In view of the results in Proposition~\ref{nonlinearterms} the mapping
$\mathcal{N}$ is a contraction if the initial values of the director field $Q$
are fixed and $T$ is small enough. This motivates the definition
\begin{align*}
 X_{T,0} = \bset{ (u,Q)\in X_T,\,Q|_{t=0}=Q_0 }\,.
\end{align*}
%
It follows that the existence of a solution is equivalent to the
assertion that the nonlinear mapping
\begin{equation*}
  \mathscr{L}(\tildeQ)=\mathcal{L}(\tildeQ)^{-1}\mathcal{N}(\tildeQ)
:X_{T,0}\to X_{T,0}
\end{equation*}
has a fixed-point. We prove the existence of such a fixed-point with Banach's fixed-point theorem applied to the closed ball $\BRXTzero$ with
\begin{align}\label{largeR}
 R >3C_{\mathcal{L}}(Q_0) \|(u_0,Q_0)\|_{Y_0}\,.
\end{align}

\medskip

\noindent
\textit{Step 2: $\mathscr{L}(\tildeQ)$ is a contraction for $T$ small enough.}
We use the estimates in
Propositions~\ref{linearized} and~\ref{nonlinearterms}. For any pair
$(u_i,Q_i)\in \BRXTzero$, $i=1, \, 2$, we obtain
\begin{equation*}
  \begin{split}
     &\|\mathcal{L}^{-1}(\tildeQ)\mathcal{N}(\tildeQ)(u_1,Q_1)-
     \mathcal{L}^{-1}(\tildeQ)\mathcal{N}(\tildeQ)(u_2,Q_2)\|_{X_T}
\\ &\qquad
\leq C_{\mathcal{L}}(Q_0)\|\mathcal{N}(\tildeQ)(u_1,Q_1)
-\mathcal{N}(\tildeQ)(u_2,Q_2)\|_{Y_T}
\\&\qquad
\leq  C_{\mathcal{L}}(Q_0)C_{\mathcal{N}}(T,R)\|(u_1-u_2,Q_1-Q_2)\|_{X_T}\,.
  \end{split}
\end{equation*}
By Proposition~\ref{nonlinearterms} (iii) we conclude that
$\mathscr{L}(\tildeQ)$ is a contraction for $T>0$ sufficiently small.

\medskip

\noindent
\textit{Step 3: $\mathscr{L}(\tildeQ)$ is a self-map for $T$ small enough.}
We infer from the bounds in Propositions~\ref{linearized} and~\ref{nonlinearterms}
and from~\bref{largeR} that for $(u,Q)\in \BRXTzero$ and $T>0$ sufficiently
small so that
\begin{align*}
 R>3C_{\mathcal{L}}(Q_0)\|(u_0,Q_0)\|_{Y_0}+M\|Q_0\|_{X_T}
\end{align*}
the estimate
\begin{align*}
\|\mathscr{L}(\tildeQ)(u,Q)\|_{X_T}
&\leq C_{\mathcal{L}}(Q_0)\|\mathcal{N}(\tildeQ)(u,Q)\|_{Y_T\times Y_0}\\
&\leq C_{\mathcal{L}}(Q_0)C_{\mathcal{N}}(T,R)\(R+M\|Q_0\|_{X_T}\)+ C_{\mathcal{L}}(Q_0,\O)\|(u_0,Q_0)\|_{Y_0}\\
&\leq 2C_{\mathcal{L}}(Q_0)C_{\mathcal{N}}(T,R)R+ \frac R3\leq R \qquad \text{for } T \text{ sufficiently small}
\end{align*}
follows.

\medskip

\noindent
\textit{Step 4: Conclusion of the proof.}
Choosing $T>0$ in the first three steps sufficiently small we obtain the assertion of the theorem
from Banach's fixed-point theorem applied to the complete metric space
$\BRXTzero$
with $R>0$ as in~\bref{largeR} and $d(z_1, z_2)= \| z_1 - z_2 \|_{X_T}$
for all $z_1$, $z_2\in \BRXTzero$. \qed
\section{Global Estimates}\label{sec:Global}
The solution of \eqref{strongformeqn} satisfies the following energy dissipation law,
\begin{equation}\label{dissipation}
  \mathcal{E}(t)+\int_0^t\mathcal{B}(\tau)d\tau\leq \mathcal{E}(0)\leq C_{\mathcal{E}}(\|u_0\|_{L^2(\O)},\|Q_0\|_{H^1(\O)})
\quad \forall t\in (0,T)\,,
\end{equation}
where $\mathcal{E}$ and $\mathcal{B}$ are given by
\begin{equation*}
  \mathcal{E}(t)=\frac 12\|u(t)\|_{L^2(\O)}^2 +\mathcal{F}(Q)
\end{equation*}
and
\begin{equation*}
  \mathcal{B}(t)=\|\nabla u(\cdot,t)\|^2_{L^2(\O)}+\|H(Q(\cdot,t))\|_{L^2(\O)}^2,
\end{equation*}
respectively,
and $\mathcal{F}$ is defined by \eqref{laudau}. 
The proof of \eqref{dissipation} can be achieved by standard energy integral method,
cf.\ Step~4 of the proof of Proposition~\ref{approximate1} in~\cite{wu} 
or~\cite[Prop 1]{PaicuZarnescuARMA2012}. By Sobolev's embedding theorem,
$H^1(\O)\hookrightarrow L^6(\O)$ and we deduce from \eqref{dissipation} that
 \begin{equation}\label{weakbound}
   \|Q\|_{L^{\infty}(0,t;H^1(\O))}+ \|u\|_{L^{\infty}(0,t;L^2(\O))}
\leq C(Q_0,u_0)\quad \text{for a.e. } t\in (0,T)\,.
 \end{equation}
 The bound for $u$ in \eqref{weakbound} is clear and  the bound for 
$Q$ can be deduced from  the estimate $f_B\geq\frac{c}4\tr Q^4 -C$ for some $C>0$, 
which follows from Young's inequality.

For $d=2$ and $\nu$ constant, we can establish the following higher-order 
energy inequality following the steps of \cite[Lemma 3.2]{wu}.

\begin{lemma}\label{uni2d}
Let $d=2$ and $\nu$ be constant.
Then the local strong solution established in Theorem \ref{maintheorem}
satisfies 
  \begin{equation*}
    \frac{d}{dt}\mathcal{B}(t)+\nu\|Au(\cdot, t)\|_{L^2(\O)}^2
+\|\nabla (H(Q(\cdot,t)))\|_{L^2(\O)}^2
\leq C\(\mathcal{B}^2(t)+\mathcal{B}(t)\)\quad \forall t\in [0,T)
  \end{equation*}
where $C$ is a constant which may depend on $\|u_0\|_{H^1(\O)},\|Q_0\|_{H^2(\O)}$ 
but is independent of $T$.
\end{lemma}

In the three-dimensional case we obtain the corresponding estimate 
following the steps of \cite[Lemma 4.1]{wu}.

\begin{lemma}\label{uni3d}
  Let $d=3$ and $\nu$ be constant. Then 
the local strong solution established in Theorem \ref{maintheorem}
satisfies
\begin{equation}\label{3duniform}
    \frac{d}{dt}\mathcal{B}(t)+\nu\|A u(\cdot, t)\|_{L^2(\O)}^2
+\|\nabla (H(Q(\cdot,t)))\|_{L^2(\O)}^2
\leq C\(\mathcal{B}^4(t)+\mathcal{B}(t)\)\quad \forall t\in [0,T)
  \end{equation}
  where $C$ is a constant which may  depend on 
$\|u_0\|_{H^1(\O)},\,\|Q_0\|_{H^2(\O)}$ but is independent of $T$.
\end{lemma}
The proofs of these two lemmas are similar to the proofs 
of~\cite[Lemma 3.2, Lemma 4.1]{wu}. We only need to show that the 
homogeneous Dirichlet condition along with the compatibility condition 
enables us to obtain the same a priori estimate. 
We present here the proof for \eqref{3duniform} and 
indicate how the proof can be adapted to the two-dimensional situation. 

\begin{proof}[Proof of Lemma \ref{uni3d}]
Without loss of generality, we can choose $\nu=2$. 
During the proof, we shall frequently use the boundary conditions
\begin{equation}\label{boundarycondition}
  u|_{(0,T)\times\p\O}=0,~Q|_{(0,T)\times\p\O}=\Delta Q|_{(0,T)\times\p\O}=H(Q)|_{(0,T)\times\p\O}=0.
\end{equation}
We test the third equation of \eqref{strongformeqn} by $\Delta H$ and the 
first equation by $Au$ (recall \eqref{stokesoper}) and integrate over $\O$
to obtain
\begin{alignat}{1}\label{Qtensor}
      \int_{\O}\partial_tQ_{\alpha\beta}\Delta H_{\alpha\beta}&+\int_\O u_{\gamma}\partial_\gamma Q_{\alpha\beta}\Delta H_{\alpha\beta}-\int_\O S_{\alpha\beta}(\nabla u,Q)\Delta H_{\alpha\beta}=\int_{\O}H_{\alpha\beta}\Delta H_{\alpha\beta},
\\\nonumber
  \int_{\O}\partial_t u_{\alpha}(Au)_{\alpha}&+\int_\O u_{\beta}\partial_\beta u_{\alpha}(Au)_{\alpha}\\\label{velocity}
&=\int_{\O}\Delta u_{\alpha} (Au)_{\alpha}+\int_\O\partial_\beta\tau_{\alpha\beta}(Q)(Au)_{\alpha}+\int_\O\partial_\beta \sigma_{\alpha\beta}(Q,Q)(Au)_{\alpha}.
\end{alignat}
For the first term in \eqref{Qtensor} we find
\begin{equation}\label{leading}
  \begin{split}
    &\int_{\O}\partial_t Q_{\alpha\beta}\Delta (H_{\alpha\beta}(Q))
=\int_{\O}\left(\Delta \partial_t Q_{\alpha\beta}\right) H_{\alpha\beta}(Q)\\
    =&\int_{\O} \p_t\(H_{\alpha\beta}(Q)-L_{\alpha\beta}(Q)\) H_{\alpha\beta}(Q)\\
    =&\int_{\O} \p_t (H_{\alpha\beta}(Q))H_{\alpha\beta}(Q)
-\p_t(L_{\alpha\beta}(Q)) H_{\alpha\beta}(Q)\\
    =&\frac 12\frac {d}{dt}\|H(Q)\|^2_{L^2(\O)}
-\int_{\O}\partial_t Q_{\gamma\delta}
\frac{\p L_{\alpha\beta}(Q)}{\p Q_{\gamma\delta}} H_{\alpha\beta}(Q)\,.
  \end{split}
\end{equation}
 The combination of \eqref{Qtensor}, \eqref{velocity} and \eqref{leading} 
gives the following energy estimate,
\begin{equation}\label{energ1}
\begin{split}
&\frac 12\frac {d}{dt}\(\|H(Q)\|^2_{L^2(\O)}
+\|A^{\frac 12} u\|_{L^2(\O)}^2\)
+\(\|\nabla H(Q)\|^2_{L^2(\O)}+\|A u\|^2_{L^2(\O)}\)\\
& \qquad-\int_{\O}\partial_t Q_{\gamma\delta}\frac{\p L_{\alpha\beta}(Q)}{\p Q_{\gamma\delta}}
 H_{\alpha\beta}(Q)
+  \int_{\O}u_{\gamma}\partial_\gamma Q_{\alpha\beta}\Delta H_{\alpha\beta}
-\int_{\O}S_{\alpha\beta}(\nabla u,Q)\Delta H_{\alpha\beta}\\
&\qquad
-\int_{\O}\partial_\beta \tau_{\alpha\beta}(Q)(Au)_{\alpha}
-\int_{\O}\partial_\beta\sigma_{\alpha\beta}(Q,Q)(Au)_{\alpha}
  +\int_{\O}u_{\beta}\partial_\beta u_{\alpha}(Au)_{\alpha}=0\,.
 \end{split}
\end{equation}
We define 
 \begin{alignat*}{3}
\mathcal{J}_0 
&\, =-\int_{\O}\partial_t Q_{\gamma\delta}
\frac{\p L_{\alpha\beta}(Q)}{\p Q_{\gamma\delta}} H_{\alpha\beta}(Q)\,, &&\\
\mathcal{J}_1
&\,=\int_{\O}u_{\gamma}\partial_\gamma Q_{\alpha\beta}\Delta H_{\alpha\beta}\,,
& \quad\mathcal{J}_3&\,=-\int_{\O}\tau_{\alpha\beta,\beta}(Q)(Au)_{\alpha}\,,\\
\mathcal{J}_2
&\,=-\int_{\O}S_{\alpha\beta}(\nabla u,Q)\Delta H_{\alpha\beta}\,,
 &\,
\quad\mathcal{J}_4&\,=-\int_{\O}\partial_\beta \sigma_{\alpha\beta}(Q,Q)(Au)_{\alpha}\,,\\
\mathcal{J}_5&\,=\int_{\O}u_{\beta}\partial_\beta u_{\alpha}(Au)_{\alpha}\,,&&
 \end{alignat*}
and rewrite \eqref{energ1} as
 \begin{equation}\label{energ2}
\frac 12\frac {d}{dt}\mathcal{B}(t)+\(\|\nabla H(Q)\|^2_{L^2(\O)}+\|A u\|^2_{L^2(\O)}\)+\sum_{k=0}^5 \mathcal{J}_k=0\,.
\end{equation}
A series of technical estimates in Appendix~\ref{proofof710}
implies in the three-dimensional situation the following estimate,
\begin{equation}\label{eslow}
  \left| \sum_{k=0}^5\mathcal{J}_k\right|
\leq \epsilon\(\|\nabla H(Q)\|^2_{L^2}+\|A u\|_{L^2}^2\)
+C(\epsilon,\|u_0\|_{H^1(\O)},\|Q_0\|_{H^2(\O)})
\mathcal{B}(t)\(\mathcal{B}^3(t)+1\)\,.
  \end{equation}
The combination of \eqref{energ2} and \eqref{eslow} implies for
$\epsilon>0$ sufficiently small the assertion of the lemma in 
formula~\eqref{3duniform}.
\end{proof}

\begin{proof}[Proof of Lemma \ref{uni2d}]
 The proof in the two-dimensional situation is done in exactly the same way. 
Instead of \eqref{eslow} one verifies
\begin{equation}\label{eslow2d}
  \left| \sum_{k=0}^5\mathcal{J}_k\right|
\leq \epsilon\(\|\nabla H(Q)\|^2_{L^2}+\|A u\|_{L^2}^2\)
+C(\epsilon,\|u_0\|_{H^1(\O)},\|Q_0\|_{H^2(\O)})\mathcal{B}(t)\(\mathcal{B}(t)+1\)\,,
  \end{equation}
cf.\ Section~\ref{proofeslow2d} for more details.
\end{proof}

\begin{appendix}
\section{Proof of \eqref{eslow}}\label{proofof710}
\subsection{Some Inequalities}
 In the sequel,  we use $C=C(u_0,Q_0)$ for a constant which may depend on 
$ \|u_0\|_{H^1(\O)}+\|Q_0\|_{H^2(\O)}$ but is independent of $t$. 
Elliptic regularity, 
the definition of $\mathcal{B}(t)$,~\eqref{weakbound} and Sobolev's embedding 
imply
\begin{equation}\label{lower}
   \begin{split}
     \|Q(t)\|^2_{H^2(\O)}&\leq C\|\Delta Q(t)\|^2_{L^2(\O)}\leq
  C\(\|H(Q(t))\|^2_{L^2(\O)}+\|L(Q(t))\|^2_{L^2(\O)}\) \\
  &\leq C(u_0,Q_0)\(\mathcal{B}(t)+1\)\text{ for a.e. } t\in (0,T)\,.
   \end{split}
\end{equation}
Let $F_k(Q)$ be a matrix-valued polynomial of $Q$ of degree less than or equal to $k$.
It follows with Sobolev's embedding theorem, Poincar\'{e}'s inequality  
and~\eqref{weakbound} that
\begin{equation}\label{poly3}
  \|F_3(Q(t))\|_{L^2(\O)}\leq C(\|\nabla Q(t)\|_{L^2(\O)}^3+1) 
\leq C(Q_0,u_0)\quad \text{for a.e. } t\in (0,T)
\end{equation}
and
\begin{equation}\label{poly2}
  \|F_2(Q(t))\|_{L^3(\O)}\leq C(\|\nabla Q(t)\|_{L^2(\O)}^2+1)
\leq C(Q_0,u_0)\quad \text{for a.e. } t\in (0,T)\,.
\end{equation}
In the following $\mathcal{J}_i$, $i=0,\ldots,5$ are defined as in Section~\ref{sec:Global}.
\subsection{Estimate of $\mathcal{J}_0$}
With the third equation in~\eqref{strongformeqn} one finds
\begin{equation*}
  \begin{split}
    -\mathcal{J}_0
=&\int_{\O}\partial_tQ_{\gamma\delta}\frac{\p L_{\alpha\beta}(Q)}{\p Q_{\gamma\delta}} H_{\alpha\beta}(Q)\\
    =&\int_{\O}\left(-u_{\theta} \partial_\theta Q_{\gamma\delta}+S_{\gamma\delta}
+ H_{\gamma\delta}\right)\frac{\p L_{\alpha\beta}(Q)}{\p Q_{\gamma\delta}}
 H_{\alpha\beta}(Q):=\sum_{i=1}^3\mathcal{I}_i\,.
      \end{split}
\end{equation*}
Note that by~\eqref{strongforcing} the derivatives
$\p L_{\alpha\beta}(Q)/ \p Q_{\gamma\delta}$ are quadratic expressions in $Q$, 
i.e., of the form $F_2(Q)$.
With \eqref{poly2}, the inequalities of H\"{o}lder, Cauchy-Schwarz, and 
Sobolev, \eqref{weakbound} and \eqref{lower} one infers
\begin{equation*}
  \begin{split}
   \left|\mathcal{I}_1\right|\leq & \|H(Q)\|_{L^6}\|u\|_{L^6}\|\nabla Q\|_{L^3}\|F_2(Q)\|_{L^3}\\
   \leq & \epsilon\|\nabla (H(Q))\|_{L^2}^2+C(\epsilon,Q_0,u_0)\|\nabla u\|_{L^2}^2\|\nabla Q\|_{L^3}^2\\
   \leq & \epsilon\|\nabla (H(Q))\|_{L^2}^2+C(\epsilon,Q_0,u_0)\|\nabla u\|_{L^2}^2
   \(\|H(Q)\|^2_{L^2}+1\)\,.
      \end{split}
\end{equation*}
A similar argument leads to the following two estimates:
\begin{equation*}
  \begin{split}
     |\mathcal{I}_2|\leq &\|\nabla u\|_{L^6}\|F_3(Q)\|_{L^2}\|H(Q)\|_{L^3}
\leq C\|\nabla u\|_{L^6}\|F_3(Q)\|_{L^2}\|\nabla (H(Q))\|^{\frac 12}_{L^2}\|H(Q)\|^{\frac 12}_{L^2}\\
   \leq  &C(u_0,Q_0)\(\|\nabla  u\|_{H^1}^2
+\|\nabla (H(Q))\|^2_{L^2}\)^{\frac 34}\|H(Q)\|^{\frac 12}_{L^2} \\
   \leq &\epsilon \(\|Au\|_{L^2}^2+\|\nabla (H(Q))\|^2_{L^2}\)
+C(\epsilon,u_0,Q_0)\|H(Q)\|^2_{L^2}\,,\\[.5ex]
  \left|\mathcal{I}_3\right|\leq &\|H(Q)\|_{L^6}\|H(Q)\|_{L^2}\|F_2(Q)\|_{L^3}
   \leq \epsilon\|\nabla (H(Q))\|^2_{L^2}+C(\epsilon,u_0,Q_0)\|H(Q)\|_{L^2}^2\,.
  \end{split}
\end{equation*}
A combination of all of the above estimates yields
\begin{equation}\label{esj0}
|\mathcal{J}_0|
\leq \epsilon\(\|\nabla (H(Q))\|^2_{L^2}
+\|A u\|_{L^2}^2\)+C(\epsilon,u_0,Q_0)\mathcal{B}(t)\(\mathcal{B}(t)+1\)\,.
\end{equation}

\subsection{Estimates of $\mathcal{J}_1+\mathcal{J}_3$}
Although there is no cancellation between the terms $\mathcal{J}_1$ and $\mathcal{J}_3$,
we combine their estimates since they share some similarities.
For $\mathcal{J}_1$, we employ \eqref{boundarycondition}, integrate by parts,
insert the definition $H_{\alpha\beta}=\Delta Q_{\alpha\beta}+L_{\alpha\beta}$
and use
\begin{align*}
 \int_{\O} u_{\gamma} (\partial_\gamma H_{\alpha\beta})H_{\alpha\beta}=0
\end{align*}
to obtain
\begin{equation}\label{j1}
  \begin{split}
    \mathcal{J}_1 =&
\int_{\O}u_{\gamma}\partial_\gamma Q_{\alpha\beta}\Delta H_{\alpha\beta}
=\int_{\O} \Delta (u_{\gamma}\partial_\gamma Q_{\alpha\beta})H_{\alpha\beta}\\
=&\int_{\O} \Delta u_{\gamma} (\partial_\gamma Q_{\alpha\beta})H_{\alpha\beta}
+\int_{\O} u_{\gamma}\(\partial_\gamma H_{\alpha\beta}
-\partial_\gamma L_{\alpha\beta}\)H_{\alpha\beta}+2\int_{\O}  \partial_\xi u_{\gamma} (\partial_\gamma\partial_\xi Q_{\alpha\beta})H_{\alpha\beta}\\
=&\int_{\O} \Delta u_{\gamma} (\partial_\gamma Q_{\alpha\beta})H_{\alpha\beta}
-\int_{\O} u_{\gamma} (\partial_\gamma L_{\alpha\beta})H_{\alpha\beta}
+2\int_{\O}  \partial_\xi u_{\gamma} (\partial_\gamma\partial_\xi Q_{\alpha\beta})
H_{\alpha\beta}\,.
\end{split}
\end{equation}
 For $\mathcal{J}_3$, we  integrate by parts and obtain
 \begin{equation*}
   \begin{split}
     \mathcal{J}_3
=&-\int_{\O}\tau_{\alpha\beta,\beta}(Q)(Au)_{\alpha}
=\int_{\O}\p_{\beta}
\left(\partial_\beta Q_{\gamma\delta}(\partial_\alpha Q_{\gamma\delta})
\right)(Au)_{\alpha}\\
=&   \int_{\O}  \Delta Q_{\gamma\delta}(\partial_\alpha Q_{\gamma\delta})(Au)_{\alpha}
+\int_{\O} \partial_\beta Q_{\gamma\delta}
(\partial_\alpha\partial_\beta Q_{\gamma\delta})(Au)_{\alpha}\,.
   \end{split}
 \end{equation*}
 The second term in the last step of the above estimate vanishes due to the orthogonality of the Helmholtz decomposition:
  \begin{equation*}
    \int_{\O} \partial_\beta Q_{\gamma\delta}(\partial_\alpha\partial_\beta Q_{\gamma\delta})(Au)_{\alpha}=-\int_{\O}\frac 12\nabla |\nabla Q|^2\cdot P_{\sigma}\Delta u=0\, .
  \end{equation*}
 Therefore we obtain
 \begin{equation}\label{j3}
 \begin{split}
     \mathcal{J}_3 =\int_{\O}  \Delta Q_{\gamma\delta}(\partial_\alpha Q_{\gamma\delta})(Au)_{\alpha}
   =\int_{\O}  \(H_{\gamma\delta} -L_{\gamma\delta}\)(\partial_\alpha Q_{\gamma\delta})(Au)_{\alpha}=\int_{\O}   H_{\gamma\delta}(\partial_\alpha  Q_{\gamma\delta})(Au)_{\alpha}\,.
 \end{split}
 \end{equation}
 In the last step, we used
 that by the definition of $L$ in \eqref{loworder} and the formula \eqref{lowerder} 
we can calculate with the help of $\partial_\alpha\tr(Q)=0$
  \begin{equation*}
   \int_{\O}L_{\gamma\delta}(A)(\partial_\alpha Q_{\gamma\delta})(Au)_{\alpha}
=-\int_{\O}\(\frac 1d\tr(Q^2)\I
+\frac{\partial f_B}{\partial Q}(Q)\):\nabla Q\cdot Au
=-\int_{\O}\nabla  (f_B(Q))\cdot Au=0\,.
 \end{equation*}
 We obtain by summation of \eqref{j1} and \eqref{j3} that
 \begin{align}\label{j1j3}
        &\left(\mathcal{J}_1+\mathcal{J}_3 \right)\\\nonumber
           &=\underbrace{\int_{\O} \Delta u_{\gamma} 
(\partial_\gamma Q_{\alpha\beta})H_{\alpha\beta}}_{\mathcal{R}_1} 
\underbrace{-\int_{\O} u_{\gamma} 
(\partial_\gamma L_{\alpha\beta})H_{\alpha\beta}}_{\mathcal{R}_2}
\underbrace{+2\int_{\O}  \partial_\xi u_{\gamma} 
(\partial_\gamma \partial_\xi Q_{\alpha\beta})H_{\alpha\beta}}_{\mathcal{R}_3}
\underbrace{+\int_{\O}   H_{\gamma\delta} 
(\partial_\alpha Q_{\gamma\delta})(Au)_{\alpha}}_{\mathcal{R}_4}\,.
 \end{align}
 Among these four terms,  $\mathcal{R}_1$ and $\mathcal{R}_4$ are of the same type. Using \eqref{lower}
 \begin{equation*}
   \begin{split}
     |\mathcal{R}_1|+|\mathcal{R}_4|&\leq \epsilon\|Au\|_{L^2}^2+C(\epsilon)\|\nabla Q\|^2_{L^6}\|H(Q)\|^2_{L^3}\\
     &\leq \epsilon\|Au\|_{L^2}^2
+C(\epsilon)\(\mathcal{B}(t)+1\)\|H(Q)\|_{L^2}\|\nabla (H(Q))\|_{L^2}\\
     &\leq \epsilon\(\|Au\|_{L^2}^2+\|\nabla (H(Q))\|_{L^2}^2\)
+C(\epsilon)\(\mathcal{B}^2(t)+1\)\|H(Q)\|_{L^2}^2\\
     &\leq \epsilon\(\|Au\|_{L^2}^2+\|\nabla (H(Q))\|_{L^2}^2\)
+C(\epsilon)\(\mathcal{B}^2(t)+1\)\mathcal{B}(t)\,.
   \end{split}
 \end{equation*}
As for $\mathcal{R}_2$ and $\mathcal{R}_3$ we obtain with \eqref{poly2} 
and \eqref{lower}
\begin{equation*}
   \begin{split}
    |\mathcal{R}_2|
     \leq &\|H(Q)\|_{L^6}\|u\|_{L^6}\|\|\nabla Q\|_{L^3}\|F_2(Q)\|_{L^3}\\
     \leq &C(u_0,Q_0)\|\nabla (H(Q))\|_{L^2}\|\nabla u\|_{L^2}\|\nabla Q\|_{L^3}\\
     \leq & \epsilon\|\nabla (H(Q))\|^2_{L^2}
+C(\epsilon,u_0,Q_0)\|\nabla u\|^2_{L^2}\(1+\mathcal{B}(t)\)\,.
   \end{split}
 \end{equation*}
 Estimates \eqref{inter2}, \eqref{poly3} and the $H^2$-estimate for the 
Dirichlet-Laplace equation lead to
 \begin{equation}\label{estimate1}
   \begin{split}
   |\mathcal{R}_3|
     \leq &2\|\nabla u\|_{L^3}\|\nabla^2 Q\|_{L^2}\|H(Q)\|_{L^6}\\
     \leq &C\|\nabla u\|^{\frac 12}_{L^2}\|\nabla^2 u\|^{\frac 12}_{L^2}
\(\|L(Q)\|_{L^2}+\|H( Q)\|_{L^2}\)\|\nabla (H(Q))\|_{L^2}\\
     \leq &C\(\|\nabla (H(Q))\|_{L^2}^2+\|\nabla^2 u\|^2_{L^2}\)^{\frac 34} 
\|\nabla u\|^{\frac 12}_{L^2}\(\|L(Q)\|_{L^2}+\|H( Q)\|_{L^2}\)\\
     \leq &\epsilon\(\|\nabla (H(Q))\|_{L^2}^2+\|Au\|^2_{L^2}\)
+ C(\epsilon,u_0,Q_0)\|\nabla u\|^2_{L^2}\(1+\mathcal{B}^2(t)\)\,.
   \end{split}
 \end{equation}
 A combination of all of the above estimates yields
 \begin{equation}\label{esj13}
  \left|\mathcal{J}_1+\mathcal{J}_3\right|
\leq \epsilon \(\|\nabla H(Q)\|_{L^2}^2+\|A u\|^2_{L^2}\)
+C(\epsilon,u_0,Q_0)\mathcal{B}(t)\(1+\mathcal{B}^2(t)\)\,.
\end{equation}

 \subsection{Estimates of $\mathcal{J}_2+\mathcal{J}_4$}
 As in the proof of Proposition \ref{approximate1} (Step 7), this is the most important part in the uniform estimate, which shows that the highest order terms  have a cancellation property. In order to show this, we shall integrate by parts several times and third order derivatives of $u$, like $\Delta \nabla u$, might appear, while in the local well-posedness, we only show the $H^2$-regularity for $u$. However, noticing that $\mathcal{J}_2+\mathcal{J}_4$ depends on $u$ (along with its derivatives) linearly, by a standard density argument, we can assume that $\Delta\nabla u\in L^2(\O)$ for almost every $t$ since in the final form that $\mathcal{J}_2+\mathcal{J}_4$ only contains $u$ and its derivatives up to order two.
For the term $\mathcal{J}_2$, by \eqref{tensor4}:
\begin{equation*}
 \mathcal{J}_2=   - \int_{\O}S_{\alpha\beta}\Delta H_{\alpha\beta}   =
 \int_{\O}\partial_\beta u_{\alpha}(Q_{\alpha\gamma}\Delta H_{\gamma\beta}-
\Delta H_{\alpha\gamma}Q_{\gamma\beta})\,.
\end{equation*}
If we integrate by parts and make use of \eqref{boundarycondition}, it follows that
\begin{equation}\label{a6}
\begin{split}
      \mathcal{J}_2
      =&\int_{\O}\Delta (\partial_\beta u_{\alpha}Q_{\alpha\gamma})H_{\gamma\beta}
-\int_{\O}\Delta (\partial_\beta u_{\alpha}Q_{\gamma\beta})H_{ \alpha\gamma}\,.
 \end{split}
 \end{equation}
 By the same trick (the skew-symmetry of the tensor $\sigma$)
that we used to prove \eqref{velicitynew}, we can replace $Au$ by $\Delta u$ in the following and obtain
 \begin{equation*}
  \mathcal{J}_4
=-\int_{\O}\partial_\beta \sigma_{\alpha\beta}(Au)_{\alpha}
=\int_{\O}\partial_\beta \sigma_{\alpha\beta}\Delta u_{\alpha}
=\int_{\O}\left(H_{\alpha\gamma}Q_{\gamma\beta}
-Q_{\alpha\gamma}H_{\gamma\beta}\right)\Delta \partial_\beta u_{\alpha}\,.
 \end{equation*}
The above two equalities imply 
 \begin{equation}\label{j2j4}
   \begin{split}
     &\mathcal{J}_2+\mathcal{J}_4\\
     &=\int_{\O}\partial_\beta u_{\alpha}\Delta Q_{\alpha\gamma}H_{\gamma\beta}-\int_{\O}  \partial_\beta u_{\alpha}\Delta Q_{\gamma\beta}H_{ \alpha\gamma}\\
     &\qquad+2\int_{\O}\left(\partial_\beta \partial_\xi u_{\alpha}(\partial_\xi Q_{\alpha\gamma})
     H_{\gamma\beta}-\partial_\beta \partial_\xi u_{\alpha}(\partial_\xi Q_{\gamma\beta})H_{ \alpha\gamma}\right)\\
     &=\int_{\O}\partial_\beta u_{\alpha}\Delta Q_{\alpha\gamma}H_{\gamma\beta}-\int_{\O}  \partial_\beta u_{\alpha}\Delta Q_{\gamma\beta}H_{ \alpha\gamma}\\
     &\qquad+2\int_{\O}\partial_\beta u_{\alpha}\(\Delta Q_{\beta\gamma}H_{\gamma\alpha}- H_{\beta\gamma}\Delta Q_{\gamma\alpha}+
     \partial_\xi Q_{\beta\gamma}(\partial_\xi H_{\gamma\alpha})-\partial_\xi H_{\beta\gamma}(\partial_\xi Q_{\gamma\alpha}) \)\\
     &=\underbrace{\int_{\O}\partial_\beta u_{\alpha}\(\Delta Q_{\beta\gamma}H_{\gamma\alpha}- H_{\beta\gamma}\Delta Q_{\gamma\alpha} \)}_{\mathcal{I}_1}+\underbrace{2\int_{\O}\partial_\beta u_{\alpha}\(
     \partial_\xi Q_{\beta\gamma}(\partial_\xi H_{\gamma\alpha})
-\partial_\xi H_{\beta\gamma}(\partial_\xi Q_{\gamma\alpha}) \)}_{\mathcal{I}_2}\,.
   \end{split}
 \end{equation}
 The part $\mathcal{I}_1$ can be estimated in the same way as 
\eqref{estimate1}. So we only consider $\mathcal{I}_2$,
  \begin{equation}\label{i2est}
  \mathcal{I}_2
      \leq C(\epsilon)\|\nabla u\|_{L^4}^2\|\nabla Q\|_{L^4}^2
+\epsilon\|\nabla (H(Q))\|_{L^2}^2.
 \end{equation}
The first term on the right-hand side can be estimated by~\eqref{inter2} 
and~\eqref{lower} by
  \begin{equation}\label{3dest}
   \begin{split}
    &\|\nabla u\|_{L^4}^2\|\nabla Q\|_{L^4}^2 \\
    &\leq  \|\nabla u\|_{L^2}^{\frac 12}\|\nabla u\|_{H^1}^{\frac 32}\|\nabla Q\|_{L^2}^{\frac 12}\|\nabla Q\|_{H^1}^{\frac 32}\\
    &\leq  C\|\nabla u\|_{L^2}^{\frac 12}\|A u\|_{L^2}^{\frac 32}\|\nabla Q\|_{L^2}^{\frac 12}\|\Delta Q\|_{L^2}^{\frac 32}\\
    &\leq \epsilon_1\|Au\|_{L^2}^2+C(\epsilon_1)\|\nabla Q\|_{L^2}^{2}\|\nabla u\|_{L^2}^{2} \|\Delta  Q\|_{L^2}^{6}\\
    &\leq \epsilon_1\|Au\|_{L^2}^2+C(\epsilon_1)\|\nabla u\|_{L^2}^{2} \(\|H(Q)\|_{L^2}^2+1\)^{3}\\
    &\leq \epsilon_1\|Au\|_{L^2}^2+C(\epsilon_1)\mathcal{B}(t) \(\mathcal{B}(t)+1\)^{3}\, .
   \end{split}
 \end{equation}
These two estimates leads to
\begin{equation}\label{esj24}
  \left| \mathcal{J}_2
+\mathcal{J}_4 \right|
\leq \epsilon\(\|\nabla H(Q)\|^2_{L^2}+\|A u\|_{L^2}^2\)
+C(\epsilon)\mathcal{B}(t)\(\mathcal{B}^3(t)+1\)\,.
\end{equation}

\subsection{Estimates of $\mathcal{J}_5$ and proof of \eqref{eslow2d}}
The estimate of $\mathcal{J}_5$ is standard and we state the result without proof.
In the three-dimensional case
\begin{equation}\label{esj5}
  |\mathcal{J}_5|
\leq C(\epsilon)\|\nabla u\|_{L^2}^4+\epsilon\|\Delta u\|^2_{L^2}
\leq C(\epsilon)\mathcal{B}^2(t)+\epsilon\|A u\|^2_{L^2}\,.
\end{equation}
The combination of \eqref{esj0}, \eqref{esj13}, \eqref{esj24} and \eqref{esj5} 
implies \eqref{eslow}.


\subsection{Modifications in the two-dimensional case}\label{proofeslow2d} 
In order to prove \eqref{eslow2d} if $d=2$ one can modify the foregoing estimates  
as follows: The term $\mathcal{J}_0$ can be estimated in the same way. 
In order to estimate $\mathcal{J}_1+\mathcal{J}_3$ we estimate 
$|\mathcal{R}_1|+|\mathcal{R}_4|$ as follows:
 \begin{equation*}
   \begin{split}
     |\mathcal{R}_1|+|\mathcal{R}_4|
&\leq \epsilon\|Au\|_{L^2}^2+C(\epsilon)\|\nabla Q\|^2_{L^4}\|H(Q)\|^2_{L^4}\\
     &\leq \epsilon\|Au\|_{L^2}^2+C(\epsilon)\|\nabla Q\|_{L^2}
\|\nabla Q\|_{H^1}\|H(Q)\|_{L^2}\|\nabla (H(Q))\|_{L^2}\\
     &\leq \epsilon\(\|Au\|_{L^2}^2+\|\nabla (H(Q))\|_{L^2}^2\)
+C(\epsilon)\(\mathcal{B}(t)+1\)\|H(Q)\|_{L^2}^2\\
     &\leq \epsilon\(\|Au\|_{L^2}^2+\|\nabla (H()Q)\|_{L^2}^2\)
+C(\epsilon)\(\mathcal{B}(t)+1\)\mathcal{B}(t)\,.
   \end{split}
 \end{equation*}
The term $\mathcal{R}_2$ can be estimated in the same way and for 
$\mathcal{R}_3$ we use 
 \begin{equation*}
   \begin{split}
   |\mathcal{R}_3|
     \leq &2\|\nabla u\|_{L^4}\|\nabla^2 Q\|_{L^2}\|H(Q)\|_{L^4}\\
     \leq &C\|\nabla u\|^{\frac 12}_{L^2}\|\nabla^2 u\|^{\frac 12}_{L^2}
\(\|L(Q)\|_{L^2}+\|H( Q)\|_{L^2}\)
\|H(Q)\|_{L^2}^{\frac12}\|\nabla (H(Q))\|_{L^2}^{\frac12}\\
     \leq &C\(\|\nabla (H(Q))\|_{L^2}+\|\nabla^2 u\|_{L^2}\) 
\|\nabla u\|^{\frac 12}_{L^2}\(\|L(Q)\|_{L^2}+\|H( Q)\|_{L^2}\)
\|H(Q)\|_{L^2}^{\frac12}\\
     \leq &\epsilon\(\|\nabla (H(Q))\|_{L^2}^2+\|Au\|^2_{L^2}\)
+ C(\epsilon,u_0,Q_0)\mathcal{B}(t)\(1+\mathcal{B}(t)\)\,.
   \end{split}
 \end{equation*}
In order to estimate $\mathcal{J}_2+\mathcal{J}_4$ one just uses 
$\|u\|_{L^4}\leq C\|u\|_{L^2}^{\frac12}\|u\|_{H^1}^{\frac12}$ instead 
of \eqref{inter33} in the estimate \eqref{3dest}.  More precisely, we only need to improve the estimate  \eqref{i2est} by adapting \eqref{3dest} to the case $d=2$:
  \begin{equation}\label{3destnew}
   \begin{split}
    &\|\nabla u\|_{L^4}^2\|\nabla Q\|_{L^4}^2 \\
    &\leq  \|\nabla u\|_{L^2}\|\nabla u\|_{H^1}\|\nabla Q\|_{L^2}\|\nabla Q\|_{H^1}\\
    &\leq  C\|\nabla u\|_{L^2}\|A u\|_{L^2}\|\nabla Q\|_{L^2}\|\Delta Q\|_{L^2}\\
    &\leq \epsilon_1\|Au\|_{L^2}^2+C(\epsilon_1)\|\nabla Q\|_{L^2}^{2}\|\nabla u\|_{L^2}^{2} \|\Delta  Q\|_{L^2}^{2}\\
    &\leq \epsilon_1\|Au\|_{L^2}^2+C(\epsilon_1)\|\nabla u\|_{L^2}^{2} \(\|H(Q)\|_{L^2}^2+1\)\\
    &\leq \epsilon_1\|Au\|_{L^2}^2+C(\epsilon_1)\mathcal{B}(t) \(\mathcal{B}(t)+1\)
   \end{split}
 \end{equation}
The estimate of 
$\mathcal{J}_5$ stays the same.

\section{Estimates for lower-order terms due to the variable viscosity}\label{appendixA}

In the following let $\mathcal{J}_1,\ldots,\mathcal{J}_5$ be as in the proof of Proposition~\ref{approximate1}.
In this appendix we prove~\eqref{total1}, that is, we verify that
for $\delta>0$ there exists a constant $C=C(\delta, \tildeQ)$
such that
\begin{subequations}
\begin{align}
%
 \label{j2est}
  \mathcal{J}_1
&\,\leq \delta \|\nabla\Delta Q\|^2_{L^2}+C(\delta,\tildeQ)\|\Delta Q\|_{L^2}^2\,,
\\ \label{j33}
  \mathcal{J}_2
&\, \leq 2\varepsilon\barg{\frac {1}{5}\|\Delta ^2Q\|^2_{L^2}
+C(\tildeQ)\|  \nabla\Delta Q \|_{L^2}^2 }\,,
\\ \label{j44}
  \mathcal{J}_3
&\, \leq
 \varepsilon\barg{ \frac {1}{5}\|\Delta ^2Q\|^2_{L^2}
+C(\tildeQ)\|\nabla\Delta Q\|^2_{L^2} }\,,
\\ \label{j5est}
  \mathcal{J}_4
&\, \leq
\delta\|\nabla\Delta Q\|^2_{L^2}+C(\tildeQ,\delta)\|G\|^2_{H^1}\,,
\\ \label{j6est}
  \mathcal{J}_5
&\, \leq\delta \|Au\|_{L^2}^2+C(\delta)\|F\|_{L^2}^2\,.
\end{align}
\end{subequations}
In the estimates,
we frequently use Poincar\'{e}'s inequality
and the second boundary condition for $Q$, namely
$\Delta Q|_{\p\O}=0$.

\medskip



\medskip

\noindent
\textit{Proof of~\bref{j2est}:}
To estimate $\mathcal{J}_1$, we apply~\eqref{boundvis},
\begin{equation*}
  \begin{split}
\mathcal{J}_1
&=-\int_{\O}
\nabla \Delta Q_{\alpha\beta}\cdot\nabla(\nu(\tildeQ))\Delta Q_{\alpha\beta}
\leq \delta\,\|\nabla \Delta Q\|^2_{L^2}
+C(\delta)\|\nabla(\nu(\tildeQ))\Delta Q \|_{L^2}^2\\
&\leq  \delta\,\|\nabla \Delta Q\|^2_{L^2}+
C(\delta,\tildeQ)\||\nabla \tildeQ| \Delta Q \|_{L^2}^2.
  \end{split}
\end{equation*}
By H\"{o}lder's inequality, Sobolev embedding theorem,
\eqref{inter3} and~\eqref{boundvis},
\begin{equation*}
\begin{split}
    \||\nabla \tildeQ| \Delta Q \|_{L^2}
&\leq \|\nabla\tildeQ\|_{L^6}\|\Delta Q\|_{L^3}
\leq C\|\nabla\tildeQ\|_{H^1}
\|\Delta Q\|_{L^2}^{\frac 12}\|\nabla\Delta Q\|_{L^2}^{\frac 12}\\
    &\leq C( \tildeQ) \|\Delta Q\|_{L^2}^{\frac 12}\|\nabla\Delta Q\|_{L^2}^{\frac 12}\\
    &\leq \delta
\|\nabla\Delta Q\|_{L^2} +C(\tildeQ,\delta)\|\Delta Q\|_{L^2}\,,
\end{split}
\end{equation*}
and the combination of these two estimates implies the assertion.

\medskip

\noindent
\textit{Proof of~\bref{j33}:}
Young's inequality and~\eqref{boundvis} allow us to infer
\begin{equation*}
 \begin{split}
    \mathcal{J}_2
& =
-2\varepsilon\int_{\O}\Delta^2  Q_{\alpha\beta}
\nabla ( \nu(\tildeQ))\cdot\nabla\Delta Q_{\alpha\beta}
\leq  2\varepsilon\barg{ \frac {1}{10}\|\Delta ^2Q\|^2_{L^2}
+\frac52\||\nabla  \nu(\tildeQ)| \nabla\Delta Q \|_{L^2}^2 }\\
&\leq 2\varepsilon\barg{ \frac {1}{10}\|\Delta ^2Q\|^2_{L^2}
+C(\tildeQ)\||\nabla  \tildeQ| \nabla\Delta Q \|_{L^2}^2}\,.
\end{split}
\end{equation*}
It remains to estimate $\||\nabla  \tildeQ| \nabla\Delta Q \|_{L^2}$.
H\"{o}lder's inequality,~\eqref{inter3}, Sobolev embedding,
\eqref{ellipticest}, and Young's inequality imply that
\begin{equation*}
 \begin{split}
     \||\nabla  \tildeQ| \nabla\Delta Q \|^2_{L^2}
&\leq \|\nabla\tildeQ\|_{L^6}^2\|\nabla\Delta Q\|^2_{L^3}
\leq C \|\nabla\tildeQ\|^2_{H^1}\|\nabla\Delta Q\|_{L^2}
\|\nabla^2\Delta Q\|_{L^2} \\
&\leq C  \|\nabla\tildeQ\|^4_{H^1}
\|\nabla\Delta Q\|_{L^2}^2
+\frac{1}{10 }\,\|\Delta^2 Q\|_{L^2} ^2\\
 &\leq C( \tildeQ) \|\nabla\Delta Q\|_{L^2}^2
+\frac{1}{10 }\,\| \Delta^2 Q\|_{L^2} ^2\,.
 \end{split}
\end{equation*}
The combination of these two estimates implies the assertion.

\medskip

\noindent
\textit{Proof of~\bref{j44}:}
Analogously to the proof of~\bref{j33} we find 
\begin{equation*}
\begin{split}
   \mathcal{J}_3
&=-\varepsilon\int_{\O} \Delta ( \nu(\tildeQ))\Delta^2  Q :  \Delta Q
\leq \varepsilon\barg{ \frac {1}{10}\|\Delta ^2Q\|^2_{L^2}
+\frac52\||\Delta ( \nu(\tildeQ))| \Delta Q \|_{L^2}^2 }\\
  &\leq \varepsilon\Barg{ \frac {1}{10}\|\Delta ^2Q\|^2_{L^2}
+C(\tildeQ )\barg{ \||\Delta  \tildeQ| \Delta Q \|_{L^2}^2
+\||\nabla\tildeQ|^2 \Delta Q \|_{L^2}^2 } }\,.
\end{split}
 \end{equation*}
With~\eqref{inter4} and \eqref{ellipticest} 
one deduces
 \begin{equation*}
   \begin{split}
     \||\Delta  \tildeQ| \Delta Q \|_{L^2}^2
&\leq\|\Delta  \tildeQ\|^2_{L^2}\| \Delta Q \|_{L^{\infty}}^2
\leq C(\tildeQ ) \| \Delta Q \|_{H^1}\| \Delta Q \|_{H^2}\\
     &\leq C(\tildeQ ) \| \nabla \Delta Q \|_{L^2}\| \Delta^2 Q \|_{L^2}\\
     &\leq \frac{1}{10  }\|\Delta ^2Q\|^2_{L^2}
+C( \tildeQ ) \| \nabla \Delta Q \|_{L^2}^2\,.
   \end{split}
 \end{equation*}
By H\"{o}lder's inequality and Sobolev's embedding theorem 
\begin{equation*}
  \||\nabla\tildeQ|^2 \Delta Q \|_{L^2}^2\leq
\|\nabla\tildeQ\|^4_{L^6}\|\Delta Q\|^2_{L^6}
\leq C( \tildeQ)\|\nabla\Delta Q\|^2_{L^2}\,.
\end{equation*}
These three estimates together imply the assertion.

\medskip

\noindent
\textit{Proof of~\bref{j5est}:}
By assumption, $G\in L^2(0,T;H^1_0(\O;\mathbb{S}_0))$ and therefore
 \begin{equation*}
   \begin{split}
     \mathcal{J}_4
& =\int_{\O}G : \Delta \barg{ \nu(\tildeQ)\Delta Q}
\leq \|G\|_{H^1}\||\nabla(\nu(\tildeQ)|\Delta Q)\|_{L^2}\\
&\leq C( \tildeQ)\|G\|_{H^1}
\barg{ \||\nabla\tildeQ|\Delta Q\|_{L^2}+\|\nabla\Delta Q\|_{L^2} }
\\     &
\leq C( \tildeQ)\|G\|_{H^1}
\barg{ \|\nabla\tildeQ\|_{H^1}\|\nabla \Delta Q\|_{L^2}
+\|\nabla\Delta Q\|_{L^2} } \\
     &\leq C( \tildeQ)\|G\|_{H^1} \|\nabla\Delta Q\|_{L^2}
\leq \delta\|\nabla\Delta Q\|^2_{L^2}+C(\tildeQ,\delta)\|G\|^2_{H^1}\,,
   \end{split}
 \end{equation*}
as asserted.

\medskip

\noindent
\textit{Proof of~\bref{j6est}:}
This follows immediately from Young's inequality,
 \begin{equation*}
   \mathcal{J}_5
=-2\int_{\O}F \cdot Au
\leq\delta \|Au\|_{L^2}^2+C(\delta)\|F\|_{L^2}^2\,.
 \end{equation*}
The estimates~\eqref{j2est},~\eqref{j33},~\eqref{j44},
\eqref{j5est} and~\eqref{j6est} imply
\begin{align}
         \sum_{i=1}^5\mathcal{J}_i
\leq &2\delta\|\nabla\Delta Q\|^2_{L^2}
+C(\delta,\tildeQ)\barg{ \|F\|_{L^2}^2+\|G\|^2_{H^1}
+\|\Delta Q\|^2_{L^2} }\\
&+\frac{3\varepsilon}{5}\|\Delta^2Q\|^2_{L^2}
     +3\varepsilon C( \tildeQ)\|\nabla\Delta Q\|^2_{L^2}
+\delta \|Au\|_{L^2}^2\,.
\end{align}

\section{Estimates for the terms $\mathcal{I}_i, i=1,2,5,6$}\label{appendixB}
In the following let $\mathcal{I}_i$ be as in the proof of Proposition~\ref{approximate1}.
The estimate of $\mathcal{I}_1$ can be handled by using~\eqref{boundvis},~\eqref{important},
~\eqref{inter3},~\eqref{ellipticest},~\eqref{stokregular},
Poincar\'{e}'s inequality and Young's inequality,
\begin{equation}\label{i1}
\begin{split}
    \mathcal{I}_1
& = \int_{\O}\diverg\sigma( \tildeQ,  Q) \cdot
(2\Lsym{u} \nabla(\nu(\tildeQ))) \\
&\leq 2\|\sigma(\tildeQ,  Q)\|_{H^1}
\|\Lsym{u} \nabla(\nu(\tildeQ))\|_{L^2}\\
&\leq C(\tildeQ)\|\Delta Q\|_{H^1}\|\tildeQ\|^{\frac 12}_{H^1}
\|\tildeQ\|^{\frac 12}_{H^2}\,\bnorm{\,|\nabla\tildeQ|\,|\Lsym{u}|\,}_{L^2}\\
&\leq C(\tildeQ)\|\Delta Q\|_{H^1}\|\nabla\tildeQ\|_{L^6}
\|\Lsym{u} \|_{L^3}\\
&\leq C(\tildeQ)\|\Delta Q\|_{H^1}\|\nabla\tildeQ\|_{L^6}
\|\nabla u\|_{L^2}^{\frac 12}\|u\|_{H^2}^{\frac 12}\\
&\leq C(\tildeQ)\|\nabla\Delta Q\|_{L^2}
\|\nabla u\|_{L^2}^{\frac 12}\|Au\|_{L^2}^{\frac 12}\\
&\leq  \delta\barg{ \|\nabla\Delta Q\|^2_{L^2}+\|Au\|^2_{L^2} }
+C(\tildeQ,\delta)\|\nabla u\|_{L^2}^{2}\,.
\end{split}
\end{equation}
We can estimate $\mathcal{I}_2$ in a similar way with $\delta, \delta_1>0$
\begin{equation}\label{i2}
\begin{split}
\mathcal{I}_2&=
- \int_{\O}\sigma(\tildeQ,  Q) : \Delta u \otimes \nabla(\nu(\tildeQ))
\\
&\leq\delta_1\|\Delta u\|^2_{L^2}
+C( \tildeQ,\delta_1)
\bnorm{\,|\sigma(\tildeQ,  Q)|\,|\nabla \tildeQ |}^2_{L^2}\\
&\leq\delta_1\|\Delta u\|^2_{L^2}
+C( \tildeQ,\delta_1)\|\sigma(\tildeQ,  Q)\|_{L^3}^2
\|\nabla \tildeQ\|^2_{L^6}\\
&\leq\delta_1\|\Delta u\|^2_{L^2}+C( \tildeQ,\delta_1)
\|\tildeQ\|^2_{L^{\infty}}\|\Delta Q\|^2_{L^3}\|\nabla \tildeQ\|^2_{L^6}\\
&\leq\delta_1\| u\|^2_{H^2}+C(\tildeQ,\delta_1)
\|\Delta Q\|_{L^2} \|\nabla\Delta Q\|_{L^2}\\
&\leq\delta\|Au\|^2_{L^2}
+\delta \|\nabla\Delta Q\|^2_{L^2}+C( \tildeQ,\delta) \|\Delta Q\|_{L^2}^2\,
\end{split}
\end{equation}
provided that $\delta_1C_{\mathcal{S}}\leq \delta$ where $C_{\mathcal{S}}$ is the constant in \eqref{stokregular}.
H\"{o}lder's inequality,~\eqref{boundvis},~\eqref{inter3} and Young's 
inequality give for $\delta>0$
 \begin{equation}\label{i5}
\begin{split}
\mathcal{I}_5
&= \int_{\O} \nu(\tildeQ )\nabla u :
\barg{\Delta \tildeQ\Delta Q -\Delta Q\Delta \tildeQ} \\
   &\leq 2\|\nabla u\|_{L^6}\|\Delta\tildeQ\|_{L^2}\|\Delta Q\|_{L^3}\\
   &\leq C(\tildeQ)\|\nabla u\|_{H^1}
\|\Delta Q\|_{L^2}^{\frac 12}\|\nabla\Delta Q\|_{L^2}^{\frac 12}\\
    &\leq C(\tildeQ)\|Au\|_{L^2}
\|\Delta Q\|_{L^2}^{\frac 12}\|\nabla\Delta Q\|_{L^2}^{\frac 12}\\
   &\leq \delta \barg{ \|Au\|^2_{L^2}+\|\nabla\Delta Q\|^2_{L^2} }
+C(\delta,\tildeQ)\|\Delta Q\|_{L^2}^{2}\,.
\end{split}
\end{equation}
The estimate of $\mathcal{I}_6$ can be carried out similarly
and leads to the upper bound with $\delta>0$
\begin{equation}\label{i6}
  \begin{split}
  \mathcal{I}_6
\leq \delta \barg{ \|Au\|_{L^2}^2+\|\nabla\Delta Q\|^2_{L^2} }
+C(\delta,\tildeQ)\|\Delta Q\|_{L^2}^{2}\,.
  \end{split}
\end{equation}
The combination of~\eqref{i1},~\eqref{i2},~\eqref{i5} and~\eqref{i6} gives
\begin{equation*}
  \mathcal{A}
\leq 8\delta \barg{ \|Au\|_{L^2}^2+\|\nabla\Delta Q\|_{L^2} ^2 }
+C(\delta,Q_0)\barg{ \|\nabla u\|_{L^2}^2+\|\Delta Q\|_{L^2}^2 }\,.
\end{equation*}

\end{appendix}


\end{document}